\documentclass[10pt, a4paper]{article}
\usepackage[margin=1in]{geometry}
\usepackage{color}
\usepackage[T1]{fontenc}
\usepackage{float}

\usepackage{graphicx}
\usepackage{subcaption}

\usepackage{amsmath}
\usepackage{amsthm, amsfonts,amssymb} 
\usepackage{lmodern}
\usepackage{mathtools}
\usepackage{tikz}
\newtheorem{thm}{Theorem}[section]
\newtheorem{lem}[thm]{Lemma}

\newtheorem{mydef}{Definition}[section]

\newtheorem{rem}{Remark}[section]
\newcommand{\bFormula}[1]{
\begin{equation} \label{#1}}
\newcommand{\eF}{\end{equation}}

\newcommand{\R}{{\mathbb{R}}}

\DeclareMathOperator{\supp}{supp}

\usepackage{caption}
\usepackage{subcaption}

\theoremstyle{remark}

\usepackage[font=footnotesize,labelfont=bf]{caption}
\numberwithin{equation}{section}

\usepackage{todonotes}

\date{}

\newcommand{\Addresses}{{
  \bigskip
  \footnotesize

  Boris Muha, \textsc{Department of Mathematics, Faculty of Science, University of Zagreb}\par\nopagebreak
  \textit{E-mail address}: \texttt{borism a$\tau$ math.hr}

\medskip
Sr\dj{}an Trifunovi\'{c}, \textsc{Department of Mathematics and Informatics, Faculty of Sciences, University of Novi Sad}\par\nopagebreak
  \textit{E-mail address}: \texttt{srdjan.trifunovic a$\tau$ dmi.uns.ac.rs}

}}

\title{Analysis of an Inelastic Contact Problem for the Damped Wave Equation}
\author{Boris Muha, Sr\dj an Trifunovi\'c}

\begin{document}

\maketitle

\begin{abstract}
In this paper, we examine the dynamic behavior of a viscoelastic string oscillating above a rigid obstacle in a one-dimensional setting, accounting for inelastic contact between the string and the obstacle. We construct a global-in-time weak solution to this problem by using an approximation method that incorporates a penalizing repulsive force of the form $\frac1\varepsilon\chi_{\{\eta<0\}} (\partial_t\eta)^-$. The weak solution exhibits well-controlled energy dissipation, which occurs exclusively during contact and is concentrated on a set of zero measure, specifically when the string moves downward. Furthermore, the velocity is shown to vanish after contact in a specific weak sense. This model serves as a simplified framework for studying contact problems in fluid-structure interaction contexts.
\end{abstract}
\textbf{Keywords and phrases:} {contact mechanics, viscoelastic string, velocity penalization}
\\${}$ \\
\textbf{AMS Mathematical Subject classification (2020):} {74M15 (Primary), 74K05, 35L05 (Secondary)}

\section{Introduction}
\subsection{Problem description}
We study the following obstacle problem for a one-dimensional viscoelastic string, which deforms only in the vertical direction. Let $l > 0$ denote the length of the string. The vertical displacement of the string is denoted by $\eta: (0,\infty) \times (0,l) \to \mathbb{R}$. The string is always positioned above or in contact with the obstacle plane ${y = 0}$, satisfying the non-penetration condition:
\begin{eqnarray}\label{obst:cond}
    \eta \geq 0, \quad \text{on } (0,\infty)\times (0,l).
\end{eqnarray}
The motion of the string is governed by the damped wave equation with an unknown contact reaction force $F_{\text{con}}$:
\begin{eqnarray}
    \partial_{tt} \eta - \partial_{txx}\eta -  \partial_{xx} \eta &= F_{con}, 
    \quad \text{on } (0,\infty)\times (0,l). \label{str:eq}
\end{eqnarray}
The string is fixed at both endpoints, with boundary conditions:
\begin{eqnarray}
    &&\eta(t,0) = \eta(t,l)=h,
\end{eqnarray}
where $h>0$ is a given constant. The problem also satisfies the energy balance:
\begin{eqnarray}\label{en:bal}
    \frac{d}{dt}\left(\frac12\int_0^l |\partial_t \eta(t)|^2dx + \frac12\int_0^l|\partial_x\eta(t)|^2dx \right)+ \int_0^l |\partial_{tx}\eta(t)|^2dx + \int_0^l D_{con}(t)dx =0,
\end{eqnarray}
where $D_{\text{con}}$ represents the energy dissipation due to contact. Finally, the problem is supplemented with the following additional assumptions:
\begin{enumerate}
    \item[(A1)] $F_{con}, D_{con}\geq 0$;
    \item[(A2)] $\text{supp}(D_{con})\subseteq\text{supp}(F_{con})  \subseteq \partial\{\eta=0\}$;
    \item[(A3)] $\text{supp}(F_{con})\cap \text{int}\{\partial_t \eta\geq 0\}=\emptyset$; 
    \item[(A4)] $\lim\limits_{t\to t_0^+} \partial_t \eta(t,x_0)=0$ for all $(t_0,x_0)\in \partial\{\eta=0\}$.
\end{enumerate}

Condition (A1) implies that the contact force $F_{\text{con}}$ always acts to push the string away from the obstacle ${y = 0}$ and that the dissipation $D_{\text{con}}$ is naturally always non-negative. The second inclusion in the condition (A2) is a version of the well-known Signorini condition (see e.g. \cite{varineq}), which asserts that the reactive force is active only when the string is in contact with the obstacle (i.e., on $\{\eta =0 \}$). Furthermore, it follows immediately from equation \eqref{str:eq} that $F_{\text{con}} = 0$ in the interior of ${\eta = 0}$, as all terms on the left-hand side of \eqref{str:eq} vanish. The first inclusion for $\text{supp}(D_{\text{con}})$ indicates that the dissipation due to contact can only happen where the force $F_{\text{con}}$ is active. Condition (A3) states that if the string is not moving downward, no contact force is present. Mechanically, this condition follows from the fact that $F_{\text{con}}$ is a reactive force. Lastly, condition (A4) ensures that the string does not bounce off the obstacle. This condition is motivated by the interpretation of the model as a simplified representation of contact in fluid-structure interactions. In such interactions, an elastic body immersed in a fluid often exhibits zero velocity after contact with a rigid surface.

Since, by (A2), the contact force $F_{\text{con}}$ is supported on $\partial\{\eta = 0\}$, which is a nowhere dense set, we expect $F_{\text{con}}$ to be a singular measure. Consequently, we adopt a weak solution framework. Furthermore, the contact set can be highly irregular and complex, and in the general case, conditions (A3) and (A4) cannot be directly formulated or proven. For example, at the level of weak solutions, condition (A4) is in general not well defined via standard trace theory. However, under certain mild assumptions about the regularity of the contact set, we can formulate and prove both conditions (A3) and (A4).

\subsection{Literature review}

The contact problem is a decades old topic that has been studied by engineers, physicists, mathematicians and many others. From a mathematical perspective, such problems present numerous challenges. One major difficulty arises from the rapid transition between contact and no-contact regimes, resulting in an abrupt switch between governing equations. This sudden change induces a reactive contact force, typically represented as a singular force—a measure that enforces the impermeability condition. Additionally, the contact surface is not known in advance, making the problem a free-boundary problem with an unknown and generally irregular boundary. This irregularity stems from the singular nature of the contact force, which often results in weak solutions. Comparing two solutions becomes particularly problematic, as their contact sets and corresponding contact forces may differ, complicating the analysis of uniqueness. Lastly, tracking the velocity at the moment of contact poses significant challenges, especially on irregular contact sets where defining such quantities precisely is nontrivial.

Let us give a brief overview of some closely related topics. The contact problem for the wave equation has been studied intensively around four decades ago. This work relies on purely hyperbolic theory in 1D and explicit calculation of contact surface and solutions. In the celebrated work by Schatzmann \cite{schatzman}, she proves the existence of a unique global weak solution to the wave obstacle problem and gives an explicit integral formula for the solution. In a recent result, this solution has been studied by Fern\'andez-Real and Figalli in \cite{FernandezFigalli} where the authors show that a certain Lipschitz norm is preserved through collision. We also mention a similar contact problem with the gluing effect was studied in \cite{CitriniMarchionna}. The contact problem for a thermoelastic von K\'{a}rm\'{a}n plate has been studied in \cite{vonKarman}, see also \cite{varineq} for more details on the dynamic plate contact problems. The stationary contact problem (known as the obstacle problem) has been a well-covered topic based on the elliptic theory and calculus of variations, which, besides well-posedness theory, has given a very valuable insight into the regularity of the contact set boundary, see \cite{Caffarelli,Friedman,Petrosyan} and the references therein. Finally, let us point out a recent progress for the contact problem in for second-gradient materials. The static case was studied in \cite{static}, which was later extended to the quasistatic case with second-grade viscoelasticity in \cite{quasistatic}, and finally the fully inertial case was studied in \cite{CGM} where a weak solution was constructed by a two-time scale approach developed in \cite{BKS}.
 
This work is primarily motivated by the contact problem in fluid-structure interaction (FSI) problems, which involve the interplay between a fluid and a solid structure that interact dynamically through a shared interface. The presence of fluid adds significant complexity to contact problems, whereas irregularities in the shape of the fluid cavity, cavity degeneration, and rapid change in velocity upon contact all create challenges for mathematical analysis, as standard tools often fail under such conditions. Additionally, the low regularity of the fluid makes it difficult to track velocity and stress at the interface during and after contact, further complicating the enforcement of physical conditions in these regions. Notable examples of such FSIs are the production of sound in the trachea, caused by the vibration of vocal cords when they come into contact (e.g. \cite{Feistauer2014}) and closing of the heart valves (e.g. \cite{Borazjani}). Although FSIs have been studied extensively for decades, modeling and analysis of contact remain open problems. Most analytical results focus on rigid body models, where challenges such as the no-contact paradox depend on boundary conditions and the smoothness of the boundary, see e.g. \cite{GVHillairet10} and references within. Additionally, in the context of fluid-structure interaction, Grandmont and Hillairet \cite{GH2016} analyzed a coupled fluid–viscoelastic beam system and proved the global existence of strong solutions without contact, providing important insight into how fluid effects can prevent collisions even in deformable structures. Furthermore, much of the research addresses dynamics prior to potential contact, with relatively few works constructing solutions that allow for contact. For rigid bodies, global-in-time weak solutions encompassing contact have been developed \cite{Feireisl03,SarkaChemetov}, and these results were recently extended to the 2D-1D incompressible fluid and beam interaction problem \cite{Casanova21} and to the full 3D-3D compressible fluid and viscoelastic body case \cite{MalteSrdjan}, while the result in \cite{breitroy} gives a condition for no contact for the 2D-1D compressible fluid and beam interaction problem. However, none of these studies analyze the system's behavior after contact occurs.
The contact problem has also been extensively studied from the numerical point of view. Since this is not the focus of the paper, we refer the interested reader to \cite{VonWahl21} and the references within.

In this paper, we propose a toy problem as a first step toward better understanding the dynamics after contact in FSI systems.

\subsection{Notation and preliminaries}\label{notation}
To establish the mathematical framework of the problem, we introduce the notation and recall relevant preliminaries that will be used throughout this work.
First, for a measurable set $Q \subset \mathbb{R}^d$, $d = 1,2$, we denote the Lebesgue space $L^p(Q)$ as the space of measurable functions $f$ defined on $Q$ such that the norm is finite:
\begin{eqnarray*}
   &&||f||_{L^p(Q)}:=  \left(\int_Q |f|^p \, dx\right)^{1/p} <\infty,  \quad 1\leq p <\infty, \\
   &&||f||_{L^\infty(Q)}:=  {\text{ess}\sup}_{Q}|f| <\infty,  \quad p =\infty.
\end{eqnarray*}
while the space of locally integrable functions $L_{loc}^p(Q)$ is defined as
\begin{eqnarray*}
    f\in L_{loc}^p(Q) \iff f\in L^p(Q') \text{ for every } Q' \subset\subset Q, \quad p \in [1,\infty].
\end{eqnarray*}
The Sobolev spaces $W^{n,p}(Q)$ are denoted as the space of functions $f \in L^p(Q)$ such that the following norm is finite:
\begin{eqnarray*}
    &&||f||_{W^{n,p}(Q)} :=  \left(\sum_{|\alpha|=0}^n \int_Q |D^\alpha f|^p \, dx \right)^{1/p} <\infty,\quad  1\leq p <\infty, \\
    && ||f||_{W^{n,\infty}(Q)} :=  \sum_{|\alpha|=0}^n {\text{ess}\sup}_{Q} |D^\alpha f|  <\infty,
\end{eqnarray*}
where $D^{\alpha}f=\frac{\partial^{|\alpha|}}{\partial_1^{\alpha_1}\dots\partial_d^{\alpha_d}} f$ is a distributional derivative, $n \in \mathbb{N}$ and $\alpha = (\alpha_1, \dots, \alpha_d)$ is a multi-index with $\alpha_i \in \mathbb{N}_0$ , and $|\alpha|=\sum_{i=1}^d\alpha_i$. In the Hilbert space case, $p=2$, we use the standard notation $H^1(Q):=W^{1,2}(Q).$ The Bochner spaces $L^p(0,T; W^{1,q}(Q))$ are defined as the set of Bochner measurable functions $f$ such that:
\begin{eqnarray*}
   &&||f||_{L^p(0,T; W^{1,q}(Q))}  := \left(\int_0^T || f(t,\cdot)||_{W^{1,q}(Q)} \, dt \right)^{1/p} < \infty, \quad 1\leq p<\infty, \\
   &&||f||_{L^\infty(0,T; W^{1,q}(Q))}  :={\text{ess}\sup}_{t\in(0,T)} || f(t,\cdot)||_{W^{1,q}(Q)} < \infty,
\end{eqnarray*}
for any $1\leq q\leq \infty$. \\

The space of distributions $\mathcal{D}'(Q)$ is defined as the dual space of $C_c^\infty(Q)$, while the space of non-negative Borel measures is denoted as $\mathcal{M}^+(Q)$, which, by the Riesz representation theorem, is isomorphic to the space of non-negative linear functionals over $C_c(Q)$. Specifically, $f \in \mathcal{M}^+(Q)$ if and only if:
\begin{eqnarray*}
    0\leq \langle f,\varphi \rangle:= \int_Q \varphi df <\infty, \quad \text{ for every non-negative } \varphi \in C_c(Q),
\end{eqnarray*}
provided that $Q$ is a locally compact Hausdorff metric space (see e.g. \cite[Theorem 2.1.4]{RudinBook}). \\

We will use the mollification of functions defined on $Q$. Let $a \in L^1(Q)$. First, we extend $a$ by zero to define a function on $\mathbb{R}^d$. For a given $\zeta \in C_c^\infty([-1,1]^2)$ such that
$\int_{[-1,1]^2} \zeta dx = 1$ and $\zeta \geq 0$, we define the time-space mollification as: 
\begin{eqnarray*}
    a_\omega(t,x):= \int_{\mathbb{R}^2} \frac1{\omega^2}\zeta\left(\frac{t-\tau}\omega, \frac{x-s}\omega\right) a(\tau,s) \, d\tau \, ds,
\end{eqnarray*}
where $\omega > 0$ is the mollification parameter. \\

For a given function $f\in L^1(Q)$, we define the (essential) support as
\begin{eqnarray*}
    \text{supp}(f):=Q\setminus \bigcup\{O \subseteq Q: O \text{ is open and } f=0 \text{ a.e. on } O\},
\end{eqnarray*}
while for a given measure $\mu \in \mathcal{M}^+(Q)$, we define its support of $\mu$ as
\begin{eqnarray*}
    \text{supp}(\mu):=\{x\in Q: \mu(O)>0 \text{ for every open neighborhood }O \text{ of } x\}.
\end{eqnarray*}

Finally, let $f: Q \to \mathbb{R}$. We decompose $f$ into its negative and positive parts as follows: $f = f^+ - f^-$, where $f^+ = \max\{0, f\}$ and $f^- = \max\{0, -f\}$.

\section{Weak formulation and main result}

\begin{mydef}\label{weak:sol}
For a given $\eta_0\in H_0^1(0,l)$ with $\eta_0\geq c>0$ and $v_0\in L^2(0,l)$, we say that the triplet $(\eta,F_{con},D_{con})$ consisting of string displacement $\eta\in L^\infty(0,\infty; H^1(0,l))\cap W^{1,\infty}(0,\infty; L^2(0,l))\cap H^1(0,\infty; H^1(0,l))$, contact force $F_{con}\in \mathcal{M}_{loc}^+([0,\infty)\times[0,l])$ and contact dissipation $D_{con}\in \mathcal{M}^+([0,\infty)\times[0,l])$ is a global weak solution to the problem $\eqref{obst:cond}-\eqref{en:bal}$ if:

\begin{enumerate}
\item $\eta \geq 0$ on $(0,\infty)\times (0,l)$ and $\eta(t,0)=\eta(t,l)=h$ for all $t\geq 0$;
\item The supports of the contact force and the contact dissipation satisfy\footnote{In the condition $\eqref{supp:prop:2}$, the set $\text{int}\{\partial_t \eta\geq 0\}$ is to be understood as $(0,T)\times(0,l)\setminus \text{supp}((\partial_t\eta)^-)$}
\begin{eqnarray}
    &&\text{supp}(D_{con}) \subseteq \text{supp}(F_{con}) \subseteq \partial\{\eta=0\} \label{supp:prop:1},\\
    &&\text{supp}(D_{con})\cap \text{int}\{\partial_t \eta\geq 0\}=\emptyset, \label{supp:prop:2}\\
    &&|\text{supp}(D_{con})|=0, \label{supp:prop:3} \\
    &&\partial_t \eta = 0 \text{ a.e. on }\text{supp}(F_{con}) \label{supp:prop:4};
\end{eqnarray}

\item $F_{con},D_{con}$ and $\partial_t \eta$ are connected through
\begin{eqnarray}
    F_{con} (-\partial_t \eta)^\omega\to D_{con} \quad \text{ in }\mathcal{D}'((0,\infty)\times(0,l)) \label{moll:prop}, 
\end{eqnarray}
where $(\partial_t \eta)^\omega$ is the time-space mollification w.r.t. parameter $\omega>0$;

\item The momentum equation
    \begin{eqnarray}
    &&\int_0^\infty\int_0^l \partial_t \eta \partial_t \varphi \, dx \, dt- \int_{0}^\infty \int_0^l \partial_{tx}\eta \partial_x\varphi \, dx \, dt - \int_0^\infty\int_0^l \partial_{x}\eta \partial_{x} \varphi \, dx \, dt\\
    &&=-\int_0^l v_0 \varphi(0,\cdot)dx- \int_0^\infty\int_0^l \varphi dF_{con} \nonumber \\
    &&\label{momeqweak}
\end{eqnarray}
holds for all $ \varphi \in C_c^\infty([0,\infty)\times (0,l))$;

\item Local energy balance
\begin{eqnarray}
    && - \frac12\int_0^\infty\int_0^l |\partial_t \eta|^2 \partial_t \varphi \, dx \, dt-\frac12\int_0^\infty\int_0^l |\partial_{x}\eta|^2  \partial_t \varphi \, dx \, dt+\int_{0}^\infty \int_0^l |\partial_{tx}\eta|^2 \varphi \, dx \, dt \nonumber\\
    &&\quad \int_0^\infty \int_0^l\varphi d D_{con}+  \int_{0}^\infty \int_0^l \partial_{tx}\eta  \partial_t\eta\partial_x\varphi \, dx \, dt + \int_0^\infty\int_0^l \partial_{x}\eta  \partial_t\eta  \partial_{x}\varphi \, dx \, dt\nonumber\\
    &&=\frac12\int_0^l |v_0|^2\varphi(0) \, dx+\frac12\int_0^l |\partial_x \eta_0|^2\varphi(0) \, dx\label{enineq}
    \end{eqnarray}
holds for all $\varphi \in C_c^\infty([0,\infty)\times(0,l))$;
\item The following renormalized momentum inequality holds for the positive part of the velocity $(\partial_t \eta)^+:=\max\{\partial_t\eta,0\}$:
\begin{eqnarray}
     &&\int_0^\infty\int_0^l b((\partial_t \eta)^+) \partial_t \varphi \, dx \, dt- \int_{0}^\infty \int_0^l \partial_{tx}\eta b'((\partial_t \eta)^+)\partial_x\varphi \, dx \, dt\nonumber\\
     &&-\int_0^\infty\int_0^l|\partial_{x} (\partial_t\eta)^+|^2 b''((\partial_t \eta)^+) \varphi  \, dx \, dt \nonumber \\
     && \quad - \int_0^\infty\int_0^l \partial_{x}\eta b'((\partial_t \eta)^+)\partial_{x} \varphi \, dx \, dt -\int_{0}^\infty \int_0^l \partial_x\eta\partial_x b'((\partial_t \eta)^+) \varphi \, dx \, dt\nonumber \\    
     &&\geq -\int_0^l b((v_0)^+) \varphi(0) \, dx\label{renorm}
\end{eqnarray}
for any non-negative $\varphi \in C_c^\infty([0,\infty)\times (0,l))$ and any $b\in C^2 ([0,\infty))$ such that\footnote{These growth conditions are not optimal and can be improved, however for us $b(x)=x^2$ is enough to deduce $\eqref{vanish:vel}$.} $b(0)=b'(0)=0$, $b'(x)\leq cx$ and $0\leq b''(x)\leq c$.
\end{enumerate}
\bigskip
\end{mydef}

\begin{rem}
Although it is unclear whether $\supp(F_{\mathrm{con}})$ has positive measure in general, condition~\eqref{supp:prop:4} permits this possibility. In particular, it allows the possibility that $\supp(F_{\mathrm{con}}) \subseteq \partial\{\eta = 0\}$ to have positive measure. Although $\partial\{\eta = 0\}$ is nowhere dense, it may still be of positive measure—for instance, $\partial\{\eta = 0\}$ could be of the form $(t_1, t_2) \times S_{\mathrm{can}}$, where $S_{\mathrm{can}}$ is a so-called \emph{fat Cantor set}, which has positive Lebesgue measure, contains no isolated points, and is nowhere dense. While such a configuration seems unlikely, it cannot be excluded due to the limited regularity of $\eta$. In a similar situation, an explicit example was constructed in the context of fluid–structure interaction in~\cite[Section~3.1]{MalteSrdjan}.
\end{rem}

We are now ready to state the first main results on global existence:

\begin{thm}[Existence of a weak solution]\label{main1}
    For a given $\eta_0\in H_0^1(0,l)$ with $\eta_0\geq c>0$ and $v_0\in L^2(0,l)$, there exists a global weak solution in the sense of Definition $\ref{weak:sol}$.
\end{thm}

Note that the condition for vanishing velocity upon contact, given in (A4), needs to be addressed separately, as it is not explicitly mentioned in Definition \ref{weak:sol}. Under the additional assumption that $\partial \{\eta = 0\}$ is locally the graph of a strictly monotonous\footnote{If $f$ is constant in time, then trivially $\eta = \partial_t \eta = 0$ on its graph. Therefore, strict monotonicity does not limit generality compared to mere monotonicity.} continuous function $f$, i.e. there exists a function $f$ such that $\{(t, f(t)): a \leq t \leq b\} \subset \partial \{\eta = 0\}$, we expect the velocity to change its sign and therefore be zero along this graph. However, since in general continuous and monotone functions satisfy only $f \in W^{1,1}(a,b)$, the trace cannot be defined in the standard sense. Indeed, first noticing that $\partial_t \eta \in L^2(0, T; H^1(0,l)) \hookrightarrow L^2(0,T; C([0,l]))$ implies that the restriction on this graph is well-defined as
\begin{eqnarray*}
    N_f: \partial_t \eta(t,x) \mapsto \partial_t \eta(t,f(t))\in L^2(a,b)
\end{eqnarray*}
where $N_f$ is the Nemytskii operator generated by $f$, and then noticing that the line element of the graph of $f$ is given by $\sqrt{1 + f'(t)^2} \in L^1(a,b)$, we cannot expect that the product $\eta(t,f(t)) \sqrt{1 + f'(t)^2}$ is integrable in general. Surprisingly, $\partial_t \eta$ can still be shown to be zero on this graph in a weak sense stated below, avoiding explicit reliance on the integral.

If, on the other hand, the function $f$ has a jump at $t_0$, then ${t_0} \times (x_0, x_1) \subseteq \partial \{\eta = 0\}$, and we address this case separately. This follows from the inequality
$\eqref{renorm}$ and relies on weak continuity in time. This is addressed in the second main result:

\begin{thm}[Characterization of the velocity upon contact]\label{main2}
Assume that $\eta\in H^1(0,T; H^1(0,l))$ with $\eta\geq 0$ is given.  

\begin{enumerate}
    \item Let $[a,b]\subseteq[0,T]$ and let $f:[a,b]\to (0,l)$ be continuous and monotonous such that its graph satisfies $\{(t,f(t)): t\in [a,b]\}\subseteq \partial\{\eta =0\}$. Then $\partial_t \eta$ is zero on the graph of $f$ in the 
    the following weak sense
    \begin{eqnarray}
    \int_{a}^b \int_{\{f(t)\leq x\leq l\}} \partial_{tx}\eta \varphi \, dx \, dt = - \int_{a}^b \int_{\{f(t)\leq x\leq l\}} \partial_t\eta \varphi_x \, dx \, dt \label{par:int:1}
\end{eqnarray}
and
\begin{eqnarray}
    \int_{a}^b \int_{\{0\leq  x\leq f(t)\}} \eta_{tx} \varphi \, dx \, dt= - \int_{a}^b \int_{\{0\leq  x\leq f(t)\}} \eta_{t} \varphi_x  \, dx \, dt \label{par:int:2}
\end{eqnarray}
for any $\varphi \in L^2(0,T; H_0^1(0,l))$;
\item Assume additionally that $\eta\in L^\infty(0,T; H_0^1(0,l))\cap W^{1,\infty}(0,T; L^2(0,l))$ satisfies the renormalized momentum inequality $\eqref{renorm}$. If $\eta = 0$ on $\{t_1\}\times [x_0,x_1]$ and $\partial_t \eta \leq 0$ on $(t_0,t_1)\times (x_0,x_1)$ for some $0<t_0<t_1<T$ and $0<x_0<x_1<l$, then 
\begin{eqnarray}
    \lim\limits_{\delta\to 0} \frac1\delta\int_{t_1}^{t_1+\delta}\int_{x_0}^{x_1}
    \partial_t \eta(t,x)\phi(x) \, dx = 0, \label{vanish:vel}
\end{eqnarray}
for any $\phi\in L^2(x_0,x_1)$

\end{enumerate}
\end{thm}

Finally, the last result concerns the contact force. In particular, assuming that the boundary $\partial\{\eta=0\}$ is sufficiently regular, we first characterize the contact force as a jump of stress or velocity, and show that the condition (A3) holds, i.e. $F_{con}$ is indeed a reactive force.

\begin{rem} Let us point out that all of our numerical data strongly indicate that the contact set is in fact quite regular. Two numerical examples that confirm this can be seen in Section \ref{numerics}.
\end{rem}

\begin{thm}[Characterization of contact force]\label{main3}
Let $\eta_0\in H_0^1(0,l)$, $v_0\in L^2(0,l)$, let $(\eta,F_{con},D_{con})$ be a weak solution in the sense of Definition $\ref{weak:sol}$. 
\begin{enumerate}
\item Let $f\in H^1(a,b)$ such that $\{(t,f(t)): a\leq t \leq b\}\subset \partial \{\eta=0 \}$. Then $F_{con}$ represents the jump in stress over the graph of $f$
\begin{eqnarray}
    &&F_{con}\Big(\{(t,f(t)): a\leq t \leq b\}\Big) \nonumber \\
    &&=-\lim\limits_{\delta\to 0}\frac1\delta\Bigg[ \int_a^b \int_{f(t)}^{f(t)+\delta} \partial_x \eta \, dx \, dt -\int_a^b \int_{f(t)-\delta}^{f(t)} \partial_x \eta \, dx \, dt \nonumber \\
    &&\quad + \int_a^b \int_{f(t)}^{f(t)+\delta} \partial_{tx} \eta \, dx \, dt- \int_a^b \int_{f(t)-\delta}^{f(t)} \partial_{tx} \eta \, dx \, dt  \Bigg],\qquad\qquad\label{iden:force}
\end{eqnarray}
and the condition (A3) holds on this graph
\begin{eqnarray}
  F_{con}\Big(\{(t,f(t)): a\leq t \leq b\}\cap \text{int}\{\partial_t \eta \geq 0\}\Big)= 0  ; \label{iden:forcea3}
\end{eqnarray}
\item If $\{t_1\}\times (x_0,x_1)\subseteq \partial \{\eta=0 \}$, then $F_{con}(t,x)=\delta_{t_1}(t)\otimes f_{con}(x)dx$ on $\{t_1\}\times (x_0,x_1)$ and $f_{con} \in L^2(x_0,x_1)$ represents the jump in velocity over $\{t_1\}\times (x_0,x_1)$:
\begin{eqnarray}
     f_{con}(x) = \lim\limits_{\delta\to 0}\frac1\delta \left[\int_{t_1}^{t_1+\delta}
    \partial_t \eta(t,x) \, dt- \int_{t_1-\delta}^{t_1}
    \partial_t \eta(t,x) \, dt \right],  \qquad \label{iden:force2}
\end{eqnarray}
for a.a. $x\in (x_0,x_1)$. Moreover, the condition $(A3)$ holds on this segment:
\begin{eqnarray}
    F_{con}\Big(\{t_1\}\times (x_0,x_1) \cap \text{int}\{\partial_t \eta \geq 0\}\Big)= 0.\label{iden:force2a3}
\end{eqnarray}

    \end{enumerate}
\end{thm}

We can represent the characterization of this theorem on the following figure $\ref{figure1}$:

\begin{figure}[h]
    \centering\makebox[\textwidth]{\includegraphics[width=\columnwidth]{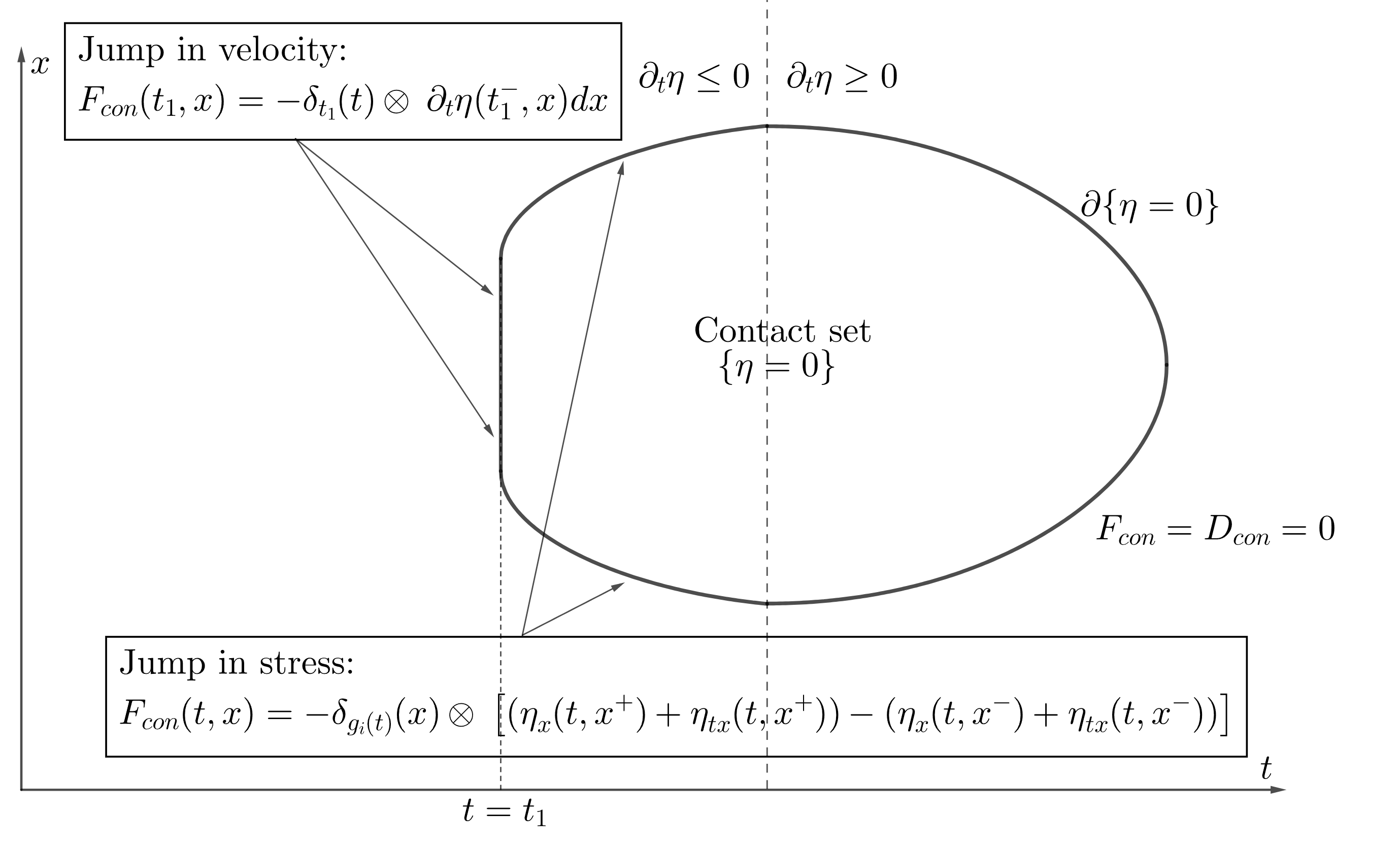}}
    \caption{An illustrative example of the contact set with characterization of contact force.}
    \label{figure1}
\end{figure}
\newpage

\section{Proof of Theorem \ref{main1}}
The solution is constructed as a limit of approximate solutions. 
\subsection{Approximate problem}
Let $\varepsilon>0$, $n\in \mathbb{N}$ and let $V_n= h+\text{span}\{ \sin(k \pi x/l) \}_{k\in \mathbb{N}}$. For a given $a>0$, define $\chi_{a}^-(x)$ to be a smooth non-increasing function such that $\chi_{a}^-(x)=1$ for $x<-a$ and $\chi_{a}^-(x)=0$ for $x>0$. We define the approximate problem as follows. Find $\eta^n \in C^2([0,T]; V_n)$ such that
\begin{eqnarray}\label{app:prob}
    &&\int_0^T\int_0^l \partial_{tt}\eta^n \varphi \, dx \, dt - \int_0^T\int_0^l \partial_{txx}\eta^n \varphi \, dx \, dt - \int_0^T\int_0^l \partial_{xx}\eta^n \varphi \, dx \, dt
    \nonumber \\
    &&+\frac1\varepsilon\int_0^T\int_0^l \chi_{\varepsilon}^-(\eta^n)\chi_{1/n}^-(\partial_t \eta^n) \partial_t \eta^n \varphi \, dx \, dt=0,
\end{eqnarray}
for all $\varphi\in C_c^\infty([0,T); V_n)$. Note that $\chi_{\varepsilon}^-(\eta)\chi_{1/n}^-(\partial_t \eta) \partial_t \eta$ approximates -$\chi_{\{\eta<0\}} (\partial_t\eta)^-$. Here, the approximate initial data $v_{0,n,\varepsilon}$ and $\eta_{0,n,\varepsilon}$ are taken as projections of smooth approximate functions $v_{0,\varepsilon}$ and $\eta_{0,\varepsilon}\geq 0$ onto $V_n$, and $v_{0,\varepsilon}\to v_n$ and $\eta_{0,\varepsilon}\to \eta_0$ in $L^2(0,l)$ and $H_0^1(0,l)$ as $\varepsilon\to 0$, respectively.

\subsection{Approximate solutions and uniform estimates}
It is easy to conclude that this problem has a unique solution $\eta^n$ by standard ODE theory, at least for $T>0$ small enough. To extend the solution to an arbitrary time, we derive the uniform estimates. By formally taking\footnote{The choice of this test function can be rigorously justified by regularizing the test function and then passing to the limit as the regularization parameter tends to zero.} $\varphi = \partial_t \eta \chi_{[0,t]}$ in $\eqref{app:prob}$ for any $t\in (0,T]$, one obtains
\begin{eqnarray}
    &&\frac12\int_0^l |\partial_t \eta^n(t)|^2\, dx \, dt + \frac12\int_0^l |\partial_x \eta^n(t)|^2\, dx \, dt + \int_0^t \int_0^l |\partial_{tx}\eta^n|^2\, dx \, dt \nonumber\\
    &&\quad + \frac1\varepsilon\int_0^t \int_0^l\chi_{\varepsilon}^-(\eta^n)\chi_{1/n}^-(\partial_t \eta^n) (\partial_t \eta^n)^2 \, dx \, dt\nonumber\\
    && =  \frac12\int_0^l |v_{0,n,\varepsilon}|^2\, dx + \frac12\int_0^l |\partial_x \eta_{0,n,\varepsilon}|^2 \, dx, \label{est1}
\end{eqnarray}
Next, we can choose $\varphi = -\partial_{xx} \eta \chi_{[0,t]}$ and $\varphi = \partial_{tt} \eta \chi_{[0,t]}$ in $\eqref{app:prob}$ for any $t\in (0,T]$, which leads to standard improved estimates for the damped wave equation
\begin{eqnarray}
        &&\int_0^t\int_0^l |\partial_{tt}\eta^n|^2\, dx \, dt+   \int_0^t\int_0^l |\partial_{xx} \eta^n|^2\, dx \, dt + \int_0^l |\partial_{tx} \eta^n(t)|^2\, dx \, dt \leq C(\varepsilon), \quad 0<t\leq T \label{impr:est}
\end{eqnarray}
where $C(\varepsilon)$ does not depend on $n$ nor $T$. Since these estimates are time-independent, we can extend the solution to $(0,\infty)$, with the following estimates:
\begin{eqnarray}
    &&||\partial_t \eta^n||_{L^\infty(0,\infty; L^2(0,l))} + ||\partial_x \eta^n||_{L^\infty(0,\infty; L^2(0,l))} + ||\partial_{tx} \eta^n||_{L^2(0,\infty; L^2(0,l))} \leq C, \label{uni:est:1} \\
    &&||\partial_{tt} \eta^n||_{L^2(0,\infty; L^2(0,l))} + ||\partial_{xx} \eta^n||_{L^2(0,\infty; L^2(0,l))} + ||\partial_{tx} \eta^n||_{L^\infty(0,\infty; L^2(0,l))} \leq C(\varepsilon).\label{uni:est:2}
\end{eqnarray}

\subsection{Passing to the limit $n\to\infty$}
Here, we can pass to the limit $n\to\infty$ in the equation $\eqref{app:prob}$ and obtain that the limiting function $\eta^\varepsilon:=\lim_n\eta^n\in H^2(0,\infty; L^2(0,l))\cap L^{\infty}(0,\infty; H^2(0,l))\cap W^{1,\infty}(0,\infty; H^1(0,l))$ satisfies
\begin{eqnarray}
    &&\int_0^\infty\int_0^l \partial_t \eta^\varepsilon \partial_t \varphi\, dx \, dt
    -\int_{0}^\infty \int_0^l \partial_{tx}\eta^\varepsilon \partial_x\varphi \, dx \, dt-\int_0^\infty\int_0^l \partial_{x}\eta^\varepsilon \partial_{x} \varphi \, dx \, dt
    \nonumber\\
    &&=-\int_0^l v_{0,\varepsilon} \varphi(0)\, dx -\int_0^\infty\int_0^l F_{con}^\varepsilon\varphi \, dx \, dt  \qquad  \label{momeqweak:eps}
\end{eqnarray}
holds for all $\varphi \in C_c^\infty([0,\infty)\times (0,l))$, where we have denoted
\begin{eqnarray*}
    F_{con}^\varepsilon:=\frac1\varepsilon \chi_{\varepsilon}^-(\eta^\varepsilon) (\partial_t \eta^\varepsilon)^-
    =-\lim\limits_{n\to\infty}\frac{1}{\varepsilon} \chi_{\varepsilon}^-(\eta^n)\chi_{1/n}^-(\partial_t \eta^n) \partial_t \eta^n.
\end{eqnarray*}
Moreover, $\eta^\varepsilon$ satisfies the energy inequality for any $t>0$
\begin{eqnarray*}
    &&\frac12\int_0^l |\partial_t \eta^\varepsilon(t)|^2\, dx + \frac12\int_0^l |\partial_x \eta^\varepsilon(t)|^2 \, dx+ \int_0^t \int_0^l |\partial_{tx}\eta^\varepsilon|^2\, dx \, dt + \int_0^t \int_0^l \frac1\varepsilon\chi_{\varepsilon}^-(\eta^\varepsilon) ((\partial_t \eta^\varepsilon)^-)^2\, dx \, dt\nonumber\\
    && \leq  \frac12\int_0^l |v_{0,\varepsilon}|^2\, dx + \frac12\int_0^l |\partial_x \eta_{0,\varepsilon}|^2\, dx.
\end{eqnarray*}

\subsection{Uniform estimates in $\varepsilon$}
First, for a given $\varphi \in C_c^\infty([0,\infty)\times [0,l])$, we can test the momentum equation $\eqref{momeqweak:eps}$ with $ \partial_t \eta^\varepsilon \varphi$ by the density argument to obtain
\begin{eqnarray}
    && - \frac12\int_0^\infty\int_0^l |\partial_t \eta^\varepsilon|^2 \partial_t \varphi \, dx \, dt-\frac12\int_0^\infty\int_0^l |\partial_{x}\eta^\varepsilon|^2  \partial_t \varphi \, dx \, dt +\int_{0}^\infty \int_0^l |\partial_{tx}\eta^\varepsilon|^2 \varphi \, dx \, dt \nonumber\\
    &&\quad+\int_0^\infty \int_0^l \frac1\varepsilon\chi_{\varepsilon}^-(\eta^\varepsilon) ((\partial_t \eta^\varepsilon)^-)^2 \varphi  \, dx \, dt +  \int_{0}^\infty \int_0^l \partial_{tx}\eta^\varepsilon  \partial_t\eta^\varepsilon\partial_x\varphi \, dx \, dt + \int_0^\infty\int_0^l \partial_{x}\eta^\varepsilon  \partial_t\eta^\varepsilon  \partial_{x}\varphi \, dx \, dt\nonumber\\
    &&=\frac12\int_0^l |v_{0,\varepsilon}|^2 \varphi(0) \, dx+\frac12\int_0^l |\partial_x \eta_{0,\varepsilon}|^2\varphi(0)\, dx.\label{enineq:eps}
    \end{eqnarray}
    
Next, from the uniform estimates $\eqref{est1}$, we have that $\partial_x \eta^\varepsilon$ is uniformly bounded in $L^\infty(0,\infty; L^2(0,l))$, so for a.a. $t>0$ one has that $||\eta(t,\cdot)||_{C^\alpha(0,l)}\leq C||\eta(t,\cdot)||_{H^1(0,l)}\leq C$ for any $\alpha \in (0,1/2)$, due to embedding of Sobolev spaces. Moreover, $\eta$ is continuous on $[0,\infty)\times[0,l]$ so this estimate holds for all $t\geq 0$. Therefore, since $\eta^\varepsilon(t,0) = \eta^\varepsilon(t,l)=h$, there is a $\delta>0$ depending only on the initial energy such that 
\begin{eqnarray}
    \eta^\varepsilon>\frac{h}2 \quad \text{ on } \quad  [0,\infty)\times([0,\delta]\cup[l-\delta,l]) \label{away:from}
\end{eqnarray}
so
\begin{eqnarray}
     F_{con}^\varepsilon = 0, \quad \text{on } [0,\infty)\times([0,\delta]\cup[l-\delta,l]). \label{away:from2}
\end{eqnarray}
By choosing a non-negative time-independent $\varphi\in C_c^\infty((0,l))$ in $\eqref{momeqweak:eps}$ such that $\varphi = 0$ on $[0,\frac\delta2]\cup[l-\frac\delta2,l]$ and $\varphi=1$ on $[\delta,l-\delta]$ and testing with $\chi_{[0,T]}\varphi$ for any $T>0$, one has
\begin{eqnarray}
    &&\int_0^T\int_0^l F_{con}^\varepsilon \, dx \, dt= \int_0^T\int_0^l F_{con}^\varepsilon \varphi\, dx \, dt \nonumber\\
    &&= \int_0^T\int_0^l \partial_t \eta^\varepsilon \underbrace{\partial_t \varphi}_{=0} dx \, dt- \int_{0}^T \int_0^l \partial_{tx}\eta^\varepsilon \partial_x\varphi\, dx \, dt -\int_0^T\int_0^l \partial_{x}\eta^\varepsilon \partial_{x} \varphi\, dx \, dt
    \nonumber\\
    &&\quad-\int_0^l v_{0,\varepsilon} \varphi(0)\, dx+\int_0^l \partial_t \eta^\varepsilon(T)\varphi(T) \, dx\nonumber\\
    &&\leq C \label{F:con:est}
\end{eqnarray}
where $C$ depends only on the initial energy and $T$. Therefore, since $F_{con}^\varepsilon \geq 0$, one has
\begin{eqnarray*}
    F_{con}^\varepsilon \text{ is uniformly bounded in } L_{loc}^1([0,\infty); L^1(0,l)).
\end{eqnarray*}
Next, in order to prove that the limit of $\eta^\varepsilon$ will be non-negative, we apply the fundamental theorem of calculus onto\footnote{We choose $(\eta^\varepsilon+\varepsilon)^-$ instead of $(\eta^\varepsilon)^-$ due to the presence of the function $\chi_\varepsilon$ in the penalization term.} $-(\eta^\varepsilon+\varepsilon)^-$ which gives us
\begin{eqnarray}
    &&(\eta^\varepsilon(t,x)+\varepsilon)^- = \underbrace{(\eta^\varepsilon(0,x)+\varepsilon)^-}_{=0, ~\text{by construction}}-\int_0^t \chi_{{\{ \eta(s,x) \leq -\varepsilon\}}} (\partial_t \eta^\varepsilon)(s,x) \, ds \nonumber \\
    &&=-\int_0^t \chi_{{\{ \eta(s,x) \leq -\varepsilon\}}} (\partial_t \eta^\varepsilon)^+(s,x) \, ds +\int_0^t \chi_{{\{ \eta(s,x) \leq -\varepsilon\}}} (\partial_t \eta^\varepsilon)^-(s,x) \, ds \nonumber \\
    &&\leq \int_0^t \chi_{{\{ \eta(s,x) \leq-\varepsilon\}}} (\partial_t \eta^\varepsilon)^-(s,x) \, ds.  \label{low:eta:est}
\end{eqnarray}
Thus one obtains 
\begin{eqnarray}\label{zero:bound}
    &&\text{ess sup}_{t\in (0,T)} \int_0^l  |(\eta^\varepsilon+\varepsilon)^-(t,x)| \, dx \nonumber \\
    &&\leq \int_0^T \int_0^l\chi_{{\{ \eta(t,x) \leq -\varepsilon\}}} (\partial_t \eta^\varepsilon)^-(s,x) \, dxds = \varepsilon \int_0^T \int_0^l F_{con}^\varepsilon \, dx \, dt \leq C \varepsilon, 
\end{eqnarray}
where $C$ depends on the initial energy and $T$. Finally, by testing $\eqref{momeqweak:eps}$ with $b'((\partial_t \eta^\varepsilon)^+) \varphi$ by the density argument, where $\varphi \in C_c^\infty([0,T)\times (0,l))$ is non-negative and $b\in C^2 ([0,\infty))$ satisfies $b(0)=b'(0)=0$, $b'(x)\leq cx$ and $0\leq b''(x)\leq c$, we obtain
\begin{eqnarray}
   &&\int_0^T\int_0^l b((\partial_t \eta^\varepsilon)^+) \partial_t \varphi \, dx \, dt- \int_{0}^T \int_0^l \partial_{tx}\eta^\varepsilon b'((\partial_t \eta^\varepsilon)^+)\partial_x\varphi \, dx \, dt \nonumber\\
     && - \int_0^T\int_0^l \partial_{x}\eta^\varepsilon b'((\partial_t \eta^\varepsilon)^+)\partial_{x} \varphi \, dx \, dt -\int_{0}^T \int_0^l \partial_x\eta^\varepsilon\partial_x b'((\partial_t \eta^\varepsilon)^+) \varphi \, dx \, dt \nonumber \\    
     &&= \int_0^T \int_0^l \underbrace{\partial_{tx}\eta^\varepsilon \partial_x b'((\partial_t \eta^\varepsilon)^+)}_{=|\partial_{x} (\partial_t\eta^\varepsilon)^+|^2 b''((\partial_t \eta^\varepsilon)^+)}\varphi\, dx \, dt  -\int_0^l b((v_{0,\varepsilon})^+) \varphi(0) \, dx \label{weak:cont:est}
\end{eqnarray}
since $F_{con}^\varepsilon b'((\partial_t \eta^\varepsilon)^+) = 0$.

\subsection{Passing to the limit $\varepsilon\to 0$}
Here, it is straightforwards to pass to the limit in the equation $\eqref{momeqweak:eps}$. In particular, one has
\begin{eqnarray*}
    \eta^\varepsilon \to \eta,& \quad \text{weakly* in } L^\infty(0,\infty; H^1(0,l)), \\
     \eta^\varepsilon \to \eta,& \quad \text{strongly in } C^{0,\alpha}([0,\infty)\times [0,l]), \\
    \partial_t \eta^\varepsilon \to \partial_t \eta,& \quad \text{weakly* in } L^\infty(0,\infty; L^2(0,l)), \\
     \partial_t \eta^\varepsilon \to \partial_t \eta,& \quad \text{weakly in } L^2(0,\infty; H_0^1(0,l)), \\
     F_{con}^\varepsilon \to F_{con},& \quad \text{weakly in } \mathcal{M}_{loc}^+([0,\infty)\times[0,l]),  
\end{eqnarray*}
for any $0<\alpha<1/2$, so $\eta$ and $F_{con}$ satisfy $\eqref{momeqweak}$, while from $\eqref{zero:bound}$ it follows that $\eta\geq 0$. Moreover, $F_{con}$ cannot be supported in $\text{int}\{\eta>0\}$ by construction, or in $\text{int}\{\eta=0\}$ since the $F_{con}$ has to be zero on this set (because on this set the left-hand side in $\eqref{momeqweak}$ vanishes), so $\text{supp}(F_{con})\subseteq \partial\{\eta=0\}$. \\

Next, since for any fixed $T>0$
\begin{eqnarray*}
    \partial_{tt} \eta^\varepsilon = \partial_{txx}\eta^\varepsilon + \partial_{xx}\eta^\varepsilon + F_{con}^\varepsilon \quad  \text{ is uniformly bounded in } L^1(0,T; H^{-1}(0,l)),
\end{eqnarray*}
by embedding of $L^1(0,l)$ into $H^{-1}(0,l)$, while $\partial_{t} \eta^\varepsilon$ is uniformly bounded in $L^2(0,T; H^1(0,l))$, one obtains by Aubin-Lions lemma that 
\begin{eqnarray}
    \partial_t \eta^\varepsilon \to \partial_t \eta \quad \text{strongly in } L_{loc}^2((0,\infty)\times (0,l)), \label{eta:t:strong}
\end{eqnarray}
since $T$ was arbitrary. \\

Let us prove the strong convergence of $\partial_x \eta^\varepsilon$. For a fixed $k>0$, define the following function
\begin{eqnarray*}
    \psi_k(t):= \begin{cases}
        1-e^{2(t-k)}, & \quad \text{for } 0\leq t< k,\\
        0, & \quad \text{for } k\leq t.
    \end{cases}
\end{eqnarray*}
Note that this is a continuous function supported on $[0,k]$ which also satisfies
\begin{eqnarray}\label{psi:property}
    -\frac12 \partial_t\psi_k(t)+\psi_k(t) = \begin{cases}
        1, & \quad \text{for } 0\leq t< k,\\
        0, & \quad \text{for } k\leq t.
    \end{cases}
\end{eqnarray}
Then, by testing $\eqref{momeqweak:eps}$ with $(\eta^\varepsilon-h) \psi_k$, one has
    \begin{eqnarray}
    &&\int_0^k\int_0^l |\partial_t \eta^\varepsilon|^2\psi_k \, dx \, dt + \int_0^k\int_0^l \partial_t \eta^\varepsilon (\eta^\varepsilon-h) \partial_t\psi_k \, dx \, dt \nonumber \\
    && \quad  -\int_{0}^k \int_0^l \partial_{tx}\eta^\varepsilon \partial_x\eta^\varepsilon \psi_k \, dx \, dt -\int_0^k\int_0^l |\partial_{x}\eta^\varepsilon|^2 \psi_k \, dx \, dt
    \nonumber\\
    &&=-\int_0^l \psi_k(0)|v_{0,\varepsilon}|^2 \, dx -\int_0^k\int_0^l F_{con}^\varepsilon(\eta^\varepsilon -h)\psi_k \, dx \, dt. \qquad  \label{etastrong}
\end{eqnarray}
We can now transform the third and fourth terms
\begin{eqnarray*}
    &&\int_{0}^k \int_0^l \partial_{tx}\eta^\varepsilon \partial_x\eta^\varepsilon \psi_k \, dx \, dt + \int_0^k\int_0^l |\partial_{x}\eta^\varepsilon|^2\psi_k \, dx \, dt \\
    &&=\frac12\int_{0}^k \int_0^l \frac{d}{dt}|\partial_{x}\eta^\varepsilon|^2 \psi_k \, dx \, dt + \int_0^k\int_0^l \partial_{x}\eta^\varepsilon \partial_x\eta^\varepsilon \psi_k  \, dx \, dt\\
   && = \int_{0}^k \int_0^l |\partial_{x}\eta^\varepsilon|^2 
    \left(-\frac12\partial_t\psi_k + \psi_k\right) \, dx \, dt + \int_0^l \psi_k(0)|\partial_x \eta_{0,\varepsilon}|^2 \, dx\\
    &&=\int_{0}^k\int_0^l |\partial_{x}\eta^\varepsilon|^2 \, dx \, dt
     + \int_0^l \psi_k(0)|\partial_x \eta_{0,\varepsilon}|^2  \, dx
\end{eqnarray*}
by $\eqref{psi:property}$, so passing to the limit $\varepsilon\to 0$ in $\eqref{etastrong}$ gives us
\begin{eqnarray*}
    &&\lim\limits_{\varepsilon\to 0}\int_{0}^k\int_0^l |\partial_{x}\eta^\varepsilon|^2 \, dx \, dt\\    
    \nonumber\\
    &&=\int_0^k\int_0^l |\partial_t \eta|^2\psi_k \, dx \, dt
    +\int_0^k\int_0^l \partial_t \eta (\eta-h) \partial_t \psi_k \, dx \, dt \\
    &&\quad +\int_0^l \psi_k(0)|v_{0,\varepsilon}|^2 \, dx -\int_0^k\int_0^l F_{con} h \psi_k \, dx \, dt -\int_0^l \psi_k(0)|\partial_x \eta_{0}|^2 \, dx \qquad  
\end{eqnarray*}
by using $\eqref{eta:t:strong}$ and the fact that $\eta=0$ on $\supp(F_{con})$. We can compare this identity with $\eqref{momeqweak}$ tested with $(\eta-h) \psi_k$ to obtain
\begin{eqnarray*}
    \lim\limits_{\varepsilon\to 0}\int_{0}^k\int_0^l |\partial_{x}\eta^\varepsilon|^2 \, dx \, dt  = \int_{0}^k\int_0^l |\partial_{x}\eta|^2 \, dx \, dt .
\end{eqnarray*}
Since this holds for any finite arbitrary $k>0$, one concludes that 
\begin{eqnarray}
    \partial_x \eta^\varepsilon \to \partial_x \eta, \quad \text{strongly in } L_{loc}^2(0,\infty; L^2(0,l)). \label{eta:x:strong}
\end{eqnarray}

\bigskip

Next, we prove that $(\partial_t \eta)^+$ satisfies the renormalized momentum inequality \eqref{renorm}.
The weak lower semicontinuity of convex superposition operators (see for example \cite[Theorem 10.20]{FeNobook}) applied to the sequence $(\partial_{x} (\partial_t\eta^\varepsilon)^+,b''((\partial_t \eta^\varepsilon)^+))$ gives us
\begin{eqnarray*}
   \int_0^T\int_0^l|\partial_{x} (\partial_t\eta)^+|^2 b''((\partial_t \eta)^+) \varphi \, dx \, dt \leq \lim\limits_{\varepsilon\to 0} \int_0^T\int_0^l|\partial_{x} (\partial_t\eta^\varepsilon)^+|^2 b''((\partial_t \eta^\varepsilon)^+) \varphi \, dx \, dt
\end{eqnarray*}
for any non-negative $\varphi \in C_c^\infty((0,\infty)\times (0,l))$, which combined with the first strong convergence $\eqref{eta:t:strong}$ and the continuity of the function $x\mapsto x^+$ allows us to pass to the limit in the inequality $\eqref{weak:cont:est}$ to obtain $\eqref{renorm}$. \\

It remains to show that $(\eta,F_{con},D_{con})$ satisfy $\eqref{supp:prop:1}-\eqref{moll:prop}$. First, let us show that $\partial_t \eta = 0$ a.e. on $\text{supp} (F_{con})$. Since $\partial_t \eta\in L^\infty(0,\infty; L^2(0,l))$, one has that $\partial_t \eta\in L^2((0,T)\times(0,l))$ for any $T>0$ so $\partial_t \eta(\cdot, x_0) \in L^2(0,T)$ for a.a. $x_0\in (0,l)$. Fix such a $x_0$. Then, $\eta(\cdot, x_0) \in H^1(0,T)$, and since $\eta$ reaches its minimum on $\{ \eta=0\}$, one obtains that $\partial_t \eta(t,x_0) = 0$ a.e. on $\{t \in (0,T): \eta(t, x_0)=0\}$. Since this holds for a.a. $x_0 \in (0,l)$ and since $T>0$ was arbitrary, one obtains that $\partial_t \eta = 0$ a.e. on $\{\eta = 0\} \supseteq \text{supp} (F_{con})$.

\bigskip

Next, since $\eta > 0$ and consequently $F_{con} = 0$ near the boundary of $(0, \infty) \times (0,l)$ due to $\eta_0 \geq c > 0$ and $\eqref{away:from}$, let $\varphi \in C_c^\infty([0, \infty) \times (0,l))$ be such that $\varphi = 0$ sufficiently far from the boundary. By testing the momentum equation with $(\partial_t \eta)^\omega \varphi$, $\omega > 0$, and passing to the limit as $\omega \to 0$, we conclude that the energy balance $\eqref{enineq}$ holds near the boundary.

To analyze the energy balance away from the boundary, we fix a ball $B \subset (0, \infty) \times (0,l)$. The idea is to compare the mollification limit $\omega \to 0$ with the penalization limit $\varepsilon \to 0$ in the energy balance equation. Let $\varphi \in C_c^\infty(B)$. Testing the momentum equation $\eqref{momeqweak}$ with $(\partial_t \eta)^\omega \varphi$ and taking the limit as $\omega \to 0$, we obtain

\begin{eqnarray}
    && \int_B|\partial_{tx}\eta|^2\varphi\, dx \, dt +\int_B \varphi dD_{con} \nonumber\\
    &&= \frac12\int_B |\partial_t \eta|^2 \partial_t \varphi \, dx \, dt+\frac12\int_B |\partial_{x}\eta|^2  \partial_t \varphi\, dx \, dt - \int_B\partial_{tx}\eta  \partial_t\eta\partial_x\varphi\, dx \, dt  - \int_B\partial_{x}\eta  \partial_t\eta  \partial_{x}\varphi \, dx \, dt   \nonumber \\ && \label{conv2}
    \end{eqnarray}
where $D_{con}$ stands for the following weak limit
\begin{eqnarray*}
    \int_B \varphi dD_{con} :=-\lim\limits_{\omega\to 0}\int_B F_{con} (\partial_t \eta)^\omega \varphi \, dx \, dt, 
\end{eqnarray*}
and $(\partial_t \eta)^\omega$ is mollification w.r.t. parameter $\omega>0$ defined precisely in Section $\ref{notation}$. We aim to identify the limit of $F_{con} (\partial_t\eta)^\omega$ on $B$ and show that it is a non-negative measure, and to do this we will show that $|\partial_{tx}\eta^\varepsilon|^2$ can only concentrate on $\text{supp}(D_{con})$ while outside of it will converge strongly in $L^1$. Since $\text{supp}(D_{con})$ is closed by definition, for any $(t,x)\in B$ such that $(t,x)\not\in \text{supp}(D_{con})$, there is a ball $B\supseteq B'\ni (t,x)$ such that $B'\cap \text{supp}(D_{con})=\emptyset$ so for any non-negative $\varphi \in C_c^\infty(B')$ one has that
\begin{eqnarray}
    && \int_{B'} |\partial_{tx}\eta|^2\varphi\, dx \, dt \nonumber\\
    &&= \frac12\int_{B'} |\partial_t \eta|^2 \partial_t \varphi \, dx \, dt+\frac12\int_{B'} |\partial_{x}\eta|^2  \partial_t \varphi \, dx \, dt - \int_{B'} \partial_{tx}\eta  \partial_t\eta\partial_x\varphi\, dx \, dt  - \int_{B'} \partial_{x}\eta  \partial_t\eta  \partial_{x}\varphi \, dx \, dt, \nonumber \\ && \label{conv3}
    \end{eqnarray}
while testing $\eqref{momeqweak:eps}$ with $\partial_t \eta^\varepsilon\varphi$ by the density argument and passing to the limit $\varepsilon\to 0$ implies
\begin{eqnarray}
    0 &\leq& \lim\limits_{\varepsilon\to 0} \int_{B'} |\partial_{tx}\eta^\varepsilon|^2\varphi \, dx \, dt +\lim\limits_{\varepsilon\to 0} \int_{B'}  F_{con}^\varepsilon (\partial_t \eta^\varepsilon)^-\varphi \, dx \, dt \nonumber\\
    &=& \frac12 \int_{B'} |\partial_t \eta|^2 \partial_t \varphi \, dx \, dt+\frac12 \int_{B'} |\partial_{x}\eta|^2  \partial_t \varphi \, dx \, dt-  \int_{B'} \partial_{tx}\eta  \partial_t\eta\partial_x\varphi \, dx \, dt -  \int_{B'} \partial_{x}\eta  \partial_t\eta  \partial_{x}\varphi \, dx \, dt,\nonumber\\
 \label{conv1}
    \end{eqnarray}
where we used the strong convergences $\eqref{eta:t:strong}$ and $\eqref{eta:x:strong}$. Comparing $\eqref{conv3}$ and $\eqref{conv1}$ implies that 
\begin{eqnarray*}
      \lim\limits_{\varepsilon\to 0} \int_{B'} |\partial_{tx}\eta^\varepsilon|^2\varphi \, dx \, dt +\lim\limits_{\varepsilon\to 0} \int_{B'}  F_{con}^\varepsilon (\partial_t \eta^\varepsilon)^-\varphi \, dx \, dt =  \int_{B'}|\partial_{tx}\eta|^2\varphi \, dx \, dt .
\end{eqnarray*}
However, since
\begin{eqnarray*}
    \int_{B'}|\partial_{tx}\eta|^2\varphi \, dx \, dt \leq \lim\limits_{\varepsilon\to 0} \int_{B'}|\partial_{tx}\eta^\varepsilon|^2\varphi \, dx \, dt  ,
\end{eqnarray*}
by the weak lower semicontinuity, this implies that $\lim\limits_{\varepsilon\to 0} \int_{B'}  F_{con}^\varepsilon (\partial_t \eta^\varepsilon)^-\varphi \, dx \, dt =0$ and $\lim\limits_{\varepsilon\to 0} \int_{B'} |\partial_{tx}\eta^\varepsilon|^2\varphi  \, dx \, dt =   \int_{B'} |\partial_{tx}\eta|^2\varphi  \, dx \, dt $. Therefore, since this holds for any $B'\subseteq B$ such that $B'\cap \text{supp}(D_{con})=\emptyset$ and any non-negative $\varphi \in C_c^\infty(B')$, we can conclude that there exists a measure $\sigma \in \mathcal{M}^+([0,\infty)\times[0,l])$ such that for any $\varphi \in C_c^\infty(B)$ it holds
\begin{eqnarray}
    \lim\limits_{\varepsilon\to 0}\int_B |\partial_{tx}\eta^\varepsilon|^2\varphi \, dx \, dt +\lim\limits_{\varepsilon\to 0}\int_B  F_{con}^\varepsilon (\partial_t \eta^\varepsilon)^-\varphi \, dx \, dt =: \int_B |\partial_{tx}\eta|^2\varphi \, dx \, dt + \int_B \varphi d \sigma. \label{pomocna}
\end{eqnarray}
Now, by using $\eqref{pomocna}$ in $\eqref{conv1}$, one has
\begin{eqnarray*}
     &&\int_B |\partial_{tx}\eta|^2\varphi \, dx \, dt + \int_B \varphi d \sigma\\
     &&= \frac12\int_B |\partial_t \eta|^2 \partial_t \varphi \, dx \, dt+\frac12\int_B |\partial_{x}\eta|^2  \partial_t \varphi \, dx \, dt- \int_B \partial_{tx}\eta  \partial_t\eta\partial_x\varphi \, dx \, dt - \int_B \partial_{x}\eta  \partial_t\eta  \partial_{x}\varphi \, dx \, dt
\end{eqnarray*}
and comparing it to $\eqref{conv2}$, we obtain that $D_{con}=\sigma$ on $B$. Since $B$ was arbitrary, one obtains that $D_{con} = \sigma \in \mathcal{M}^+([0,\infty)\times[0,l])$. This also concludes the proof of $\eqref{enineq}$. \\

Let us now show that $|\text{supp}(D_{con})|=0$. Fix $T>0$. Since passing to the mollification limit $\omega\to0$ gives $(\partial_t \eta)^\omega \to \partial_t \eta$ a.e., by using the theorem of Egorov, for every $\delta>0$, there exists a set $A_\delta$ such that $|(0,T)\times (0,l)\setminus A_\delta| < \delta$ and $(\partial_t \eta)^\omega\to \partial_t \eta$ uniformly on $A_\delta$. Since $\partial_t \eta = 0$ a.e. on $\text{supp} (F_{con})\cap A_\delta$, this means that $|(\partial_t \eta)^\omega|\leq C_{\delta,\omega}$ a.e. on $\text{supp} (F_{con})\cap A_\delta$, where $C_{\delta,\omega}\to 0$ as $\omega\to 0$ for any fixed $\delta>0$. Therefore
\begin{eqnarray*}
    \left|\int_{\text{supp} (F_{con})\cap A_\delta}  F_{con} (\partial_t \eta)^\omega \, dx \, dt \right| \leq C_{\delta,\omega}  \int_{(0,T)\times(0,l)} dF_{con} \leq C_{\delta,\omega} C
\end{eqnarray*}
so
\begin{eqnarray*}
    \lim\limits_{\omega\to 0}\int_{\text{supp} (F_{con})\cap A_\delta} F_{con} (\partial_t \eta)^\omega \, dx \, dt = 0.
\end{eqnarray*}
Passing to the limit $\delta\to 0$ and then $T\to \infty$, the conclusion follows. \\

Finally, in order to show $\text{supp}(D_{con})\cap \text{int}\{\partial_t \eta\geq 0\}=\emptyset$, fix a non-negative $\varphi \in C_c^\infty(\text{int}\{\partial_t \eta\geq 0\})$ and $\omega_0:=\text{dist}(\text{supp} (\varphi),\partial\{\partial_t \eta\geq 0\})$. Then $(\partial_t\eta)^\omega\varphi \geq 0$ for all $\omega<\omega_0$, so
\begin{eqnarray*}
   0\leq \int_{0}^T \int_0^l \varphi d D_{con}= -  \lim\limits_{\omega_0>\omega\to 0}\int_0^T \int_0^l F_{con} (\partial_t\eta)^\omega\varphi \, dx \, dt  \leq 0
\end{eqnarray*}
since $F_{con} \geq 0$, which finishes the proof.

\section{Proof of Theorem $\ref{main2}$}

We introduce the notion of zero trace, which will play a useful role in our analysis. This concept allows us to clearly state and rigorously prove the second main result of this paper.

\begin{mydef}
Let $f:[a,b]\to (0,l)$ such that $f \in W^{1,1}(a,b)$ and let $u\in L^2(a,b; H^1(0,l))$. We say that $u$ has a zero trace on the graph of $f$ with notation $u\in \gamma_{f}^0$ if
    \begin{eqnarray}
         \lim\limits_{\omega\to 0}  \int_f u_\omega\varphi:= \lim\limits_{\omega\to 0}\int_a^b u_\omega(t,f(t))\varphi(t,f(t)) \sqrt{1+f'(t)^2}dt  = 0 \label{aux:prop}
\end{eqnarray}
for any non-negative $\varphi \in C_c^\infty([a,b]\times (0,l))$ and any set of functions $\{u_\omega\}_{\omega>0}$ such that $u_\omega\in C^\infty([a,b]\times [0,l])$ and $u_\omega \to u$ in $L^2(a,b; H^1(0,l))$ as $\omega\to 0$.
\end{mydef}

\begin{rem}
    Note that we require test function to be non-negative. This is used in order to track the change of sign across the graph of $f$. Later, this condition will be removed.
\end{rem}

We now show that this property is equivalent to the zero trace in the following weak sense:
\begin{lem}\label{equiv:lem}
Let $f:[a,b]\to (0,l)$ such that $f \in W^{1,1}(a,b)$ and let $u\in L^2(a,b; H^1(0,l))$. Then $u\in \gamma_{f}^0$ if and only if  
\begin{eqnarray}
    \int_{a}^b \int_{\{f(t)\leq x\leq l\}} u_{x} \varphi \, dx \, dt = - \int_{a}^b \int_{\{f(t)\leq x\leq l\}}u \varphi_x \, dx \, dt \label{prop1}
\end{eqnarray}
if and only if
\begin{eqnarray}
    \int_{a}^b \int_{\{0\leq  x\leq f(t)\}} u_x \varphi \, dx \, dt = - \int_{a}^b \int_{\{0\leq  x\leq f(t)\}}u \varphi_x \, dx \, dt \label{prop2}
\end{eqnarray}
for any non-negative $\varphi \in C_c^\infty([a,b]\times (0,l))$.
\end{lem}
\begin{proof}
Let $\omega>0$ and let $u_\omega \in C^\infty([a,b]\times [0,l])$ such that $u_\omega \to u$ in $L^2(a,b; H^1(0,l))$ as $w\to 0$. By integration by parts
\begin{eqnarray*}
    \int_{a}^b \int_{\{f(t)\leq x\leq f(t)+\delta\}} \partial_x u_\omega \varphi \, dx \, dt = - \int_{a}^b \int_{\{f(t)\leq x\leq f(t)+\delta\}}u_\omega \partial_x \varphi\, dx \, dt + \int_f u_\omega\varphi
\end{eqnarray*}
and 
\begin{eqnarray*}
   \int_{a}^b \int_{\{f(t)-\delta\leq  x\leq f(t)\}} \partial_x u_\omega \varphi \, dx \, dt = - \int_{a}^b \int_{\{f(t)-\delta\leq  x\leq f(t)\}}u_\omega \partial_x \varphi \, dx \, dt + \int_f u_\omega\varphi 
\end{eqnarray*}
so by passing to the limit $\omega\to 0$ we obviously obtain the desired equivalence.
\end{proof}

\begin{lem}\label{zero:trace:lemma}
     Let $u\in L^2(a,b; H^1(0,l))$ and $f:[a,b]\to [0,l]$ with $f\in W^{1,1}(a,b)$. Assume that for any fixed non-negative $\varphi \in C^\infty([a,b]\times[0,l])$, there exists $\delta_0>0$ such that for any $0<\delta<\delta_0$
\begin{eqnarray*}
    \int_a^b \int_{\{f(t)-\delta<x<f(t)\}} u \varphi \, dx \, dt \leq o(\delta), \quad \text{and }  \quad \int_a^b \int_{\{\{f(t)<x<f(t)+\delta\}} u\varphi \, dx \, dt \geq o(\delta),
\end{eqnarray*}
where $\frac{o(\delta)}{\delta}\to 0$ as $\delta \to 0$. Then $u \in \gamma_f^0$.
\end{lem}
\begin{proof}
For a given $\delta>0$ and a fixed non-negative $\varphi \in C^\infty([a,b]\times[0,l])$, the goal is to find $\omega = \omega(\delta)>0$ such that
\begin{eqnarray*}
     \left|\int_f u_\omega\varphi \right| \leq \alpha(\delta), \quad  \text{ where } \alpha(\delta) \to 0 \text{ as } \delta\to 0.
\end{eqnarray*}
By the strong convergence $u_\omega\to u$ in $L^2(a,b; H^1(0,l))$, one also has that $u_\omega\to u$ in $L^1((a,b)\times(0,l))$ so the integrals 
\begin{eqnarray*}
    \int_a^b \int_{\{f(t)-\delta<x<f(t)\}} u_\omega\varphi \, dx \, dt \quad \text{and} \quad \int_a^b \int_{\{f(t)<x<f(t)+\delta\}}u_\omega\varphi \, dx \, dt
\end{eqnarray*}
converge as well. Therefore, there exists $\omega=\omega(\delta)>0$ such that, say 
\begin{eqnarray*}
    &&\int_a^b \int_{\{f(t)-\delta<x<f(t)\}} u_\omega\varphi \, dx \, dt - \int_a^b \int_{\{f(t)-\delta<x<f(t)\}} u\varphi \, dx \, dt \leq \delta^2, \\
    &&\int_a^b \int_{\{f(t)<x<f(t)+\delta\}} u_\omega\varphi \, dx \, dt - \int_a^b \int_{\{f(t)<x<f(t)+\delta\}} u\varphi \, dx \, dt \geq -\delta^2,
\end{eqnarray*}
so
\begin{eqnarray*}
    \int_a^b \int_{\{f(t)-\delta<x<f(t)\}} u_\omega\varphi \, dx \, dt \leq \underbrace{o(\delta)+\delta^2}_{=o(\delta)} \quad \text{and} \quad \int_a^b \int_{\{f(t)<x<f(t)+\delta\}}u_\omega\varphi \, dx \, dt \geq \underbrace{o(\delta)- \delta^2}_{=o(\delta)}.
\end{eqnarray*}
By the Mean-value theorem applied to $x\mapsto \int_{a}^b (u_\omega \varphi)(f(t)+x,t)\, dt$ we have
\begin{eqnarray*}
    \int_a^b \int_{\{f(t)-\delta<x<f(t)\}} u_\omega\varphi \, dx \, dt = \delta\int_{f-\theta_1} u_\omega\varphi \quad \text{and} \quad \int_a^b \int_{\{f(t)<x<f(t)+\delta\}} u_\omega\varphi \, dx \, dt = \delta\int_{f+\theta_2} u_\omega\varphi 
\end{eqnarray*}
where $0\leq \theta_1,\theta_2\leq \delta$, $\theta_1,\theta_2\in\R$. This implies
\begin{eqnarray}
    \int_{f-\theta_1} u_\omega\varphi \leq \frac{o(\delta)}{\delta} \quad \text{and} \quad \int_{f+\theta_2} u_\omega\varphi \geq \frac{o(\delta)}\delta \label{some:est}
\end{eqnarray}
so calculating
\begin{eqnarray*}
     &&\int_{f+\theta_2} u_\omega\varphi - \int_{f-\theta_1} u_\omega\varphi \\
     && = \int_{a}^b \int_{\{f(t)-\theta_1 \leq x \leq f(t)+\theta_2 \}} \partial_x (u_\omega\varphi) \, dx \, dt \\
     &&\leq (b-a)^{1/2} |\theta_2 +\theta_1|^{1/2} ||u_\omega \varphi||_{L^2(a,b; H^1(0,l))} \\
    && \leq C \sqrt{\delta}
\end{eqnarray*}
one obtains
\begin{eqnarray*}
    \int_{f+\theta_2} u_\omega\varphi \leq \int_{f-\theta_1} u_\omega\varphi + C\sqrt\delta \leq \frac{o(\delta)}\delta+C\sqrt{\delta}, \quad \text{ and } \quad \int_{f-\theta_1} u_\omega\varphi \geq \int_{f+\theta_2} u_\omega \varphi - C\sqrt\delta \geq \frac{o(\delta)}\delta-C\sqrt{\delta}.
\end{eqnarray*}
Consequently
\begin{eqnarray*}
    &&\int_f u_\omega \varphi = \int_{f+\theta_2}u_\omega \varphi - \int_{a}^b \int_{f(t)\leq x \leq f(t)+\theta_2} \partial_x (u_\omega\varphi) \, dx \, dt \\
    &&\leq  \int_{f+\theta_2}u_\omega\varphi + (b-a)^{1/2} \theta_2^{1/2} ||u_\omega\varphi ||_{L^2(a,b; H^1(0,l))}\\
     && \leq \frac{o(\delta)}\delta+ C\sqrt\delta
     \end{eqnarray*}
and similarly
\begin{eqnarray*}
     &&\int_f u_\omega\varphi = \int_{f-\theta_1}u_\omega\varphi + \int_{a}^b \int_{f(t)-\theta_1 \leq x \leq f(t)} \partial_x (u_\omega\varphi) \, dx \, dt \\
     &&\geq  \int_{f-\theta_1}u_\omega\varphi - C\theta_1^{1/2} \\
     &&\geq \frac{o(\delta)}\delta -C\sqrt\delta 
\end{eqnarray*}
which gives us the desired conclusion.

\end{proof}

\subsection{Statement 1}
Let $f:[a,b]\to (0,l)$ with $f\in W^{1,1}(a,b)$ be given and w.l.o.g. non-decreasing. Fix a non-negative function $\varphi \in C_c^\infty([a,b]\times (0,l))$ and $\delta_0$ small enough so that $\delta_0< f(t)<l-\delta_0$ on $[a,b]$. Then, for any $0<\delta<\delta_0$ we calculate
\begin{eqnarray*}
     &&\int_a^b \int_{\{f(t)<x<f(t)+\delta\}} \partial_t \eta \varphi \, dx \, dt = \\
     && \underbrace{\int_a^b \eta(t,f(t))\varphi(t,f(t)) \frac{f'(t)}{\sqrt{1+|f'(t)|^2}} \, dt}_{=0} - \underbrace{\int_a^b \eta(t,f(t)+\delta) \varphi(t,f(t))\frac{f'(t)}{\sqrt{1+|f'(t)|^2}} \, dt}_{\geq 0}\\
     &&\quad + \int_{f(b)}^{f(b)+\delta} \eta(b,x)\varphi(b,x)\, dx - \int_{f(a)}^{f(a)+\delta} \eta(b,x)\varphi(a,x)\, dx + \int_a^b \int_{\{f(t)<x<f(t)+\delta\}}\eta  \partial_t \varphi \, dx \, dt,
\end{eqnarray*}
since the graph of $f$ is contained in $\partial\{\eta=0\}$. Next, due to embedding
\begin{eqnarray*}
    \eta\in H^1(a,b; H^1(0,l)) \hookrightarrow C^{0,\alpha}([a,b] \times [0,l]), \quad \text{ for any } 0<\alpha<1/2,
\end{eqnarray*}
one has
\begin{eqnarray*}
    \int_{f(b)}^{f(b)+\delta} \eta(b,x)\varphi(b,x)\, dx \leq C \int_{f(b)}^{f(b)+\delta} |x-f(b)|^\alpha\,  dx \leq C \delta^{1+\alpha}
\end{eqnarray*}
since $\eta(b,f(b))=0$, and similarly for $\int_{f(a)}^{f(a)+\delta} \eta(b,x)\, dx$, while the last term is trivially bounded by $C \delta^{1+\alpha}$. Therefore, one obtains 
\begin{eqnarray*}
    \int_a^b \int_{\{f(t)<x<f(t)+\delta\}} \partial_t \eta \varphi \, dx \, dt \leq o(\delta),
\end{eqnarray*}
and analogously 
\begin{eqnarray*}
    \int_a^b \int_{\{f(t)-\delta<x<f(t)\}} \partial_t \eta \varphi \, dx \, dt \geq o(\delta) .
\end{eqnarray*}
Thus, Lemma $\ref{zero:trace:lemma}$ and Lemma $\ref{equiv:lem}$ give us that $\eqref{par:int:1}$ and $\eqref{par:int:2}$ hold for any non-negative test function $\varphi \in C_c^\infty([a,b]\times(0,l))$. By density, space of non-negative $\varphi \in L^2(a,b; H_0^1(0,l))$ is dense is the space of non-negative functions in $L^2(a,b; H_0^1(0,l))$. Therefore, if an arbitrary $\varphi\in L^2(a,b; H_0^1(0,l))$ is given, then $\eqref{par:int:1}$ and $\eqref{par:int:2}$ hold for $\varphi^+$ and $\varphi^-$ which are both non-negative and belong to $L^2(a,b; H_0^1(0,l))$, so by taking the difference of identities $\eqref{par:int:1}$ and $\eqref{par:int:2}$ tested by $\varphi^+$ and $\varphi^-$, respectively, we obtain that $\eqref{par:int:1}$ and $\eqref{par:int:2}$ hold for any test function $\varphi\in L^2(a,b; H_0^1(0,l))$.

\subsection{Statement 2}
In order to show the property $\eqref{vanish:vel}$, in $\eqref{renorm}$ we choose $b(x)=x^2$, so for any  $\varphi\in C_c^\infty((0,\infty)\times (0,l))$ one can bound 
\begin{eqnarray*}
    &&-\int_0^\infty\int_0^l ((\partial_t \eta)^+)^2 \partial_t \varphi \, dx \, dt \leq - 2\int_{0}^\infty \int_0^l \partial_{tx}\eta (\partial_t \eta)^+\partial_x\varphi \, dx \, dt+2\int_0^\infty\int_0^l|\partial_{x} (\partial_t\eta)^+|^2 \varphi  \, dx \, dt \nonumber\\
     && -2 \int_0^\infty\int_0^l \partial_{x}\eta (\partial_t \eta)^+\partial_{x} \varphi \, dx \, dt-2\int_{0}^\infty \int_0^l \partial_x\eta\partial_x (\partial_t \eta)^+ \varphi \, dx \, dt \nonumber \\    
     &&\leq C ||\varphi ||_{L^2(0,\infty; W^{1,\infty}(0,l))}.
\end{eqnarray*}
By the density argument, we can choose $\varphi(t,x) = \psi_\delta(t)\phi(x)$ where $0\leq \phi\in C_c^\infty((x_0,x_1))$ and
\begin{eqnarray*}
    \psi_\delta(t)=\begin{cases}
        0,& t\leq t_1-\delta, \\
        \frac1\delta(t-(t_1-\delta)),& t_1-\delta\leq t \leq t_1, \\
        -\frac1\delta(t-(t_1+\delta)), & t_1\leq t \leq t_1+\delta,\\
         0,& t> t_1+\delta,
    \end{cases}
\end{eqnarray*}
and then passing to the limit $\delta\to 0$ (note that $||\psi_\delta(t)\phi(x)||_{L^2(0,\infty; W^{1,\infty}(0,l))} \to 0$), one has
\begin{eqnarray}
     \limsup\limits_{\delta\to 0} \frac1\delta\int_{t_1}^{t_1+\delta}\int_{x_0}^{x_1} ((\partial_t \eta)^+(t,x))^2\phi(x) \, dx \, dt     \leq \liminf\limits_{\delta\to 0} \frac1\delta\int_{t_1-\delta}^{t_1}\int_{x_0}^{x_1}((\partial_t \eta)^+(t,x))^2\phi(x) \, dx \, dt=0, \label{pos:semi:cont}
\end{eqnarray}
since the integral
\begin{eqnarray*}
    \int_{t_1-\delta}^{t_1}\int_{x_0}^{x_1}((\partial_t \eta)^+(t,x))^2\phi(x) \, dx \, dt
\end{eqnarray*}
vanishes for $\delta$ small enough by assumption. On the other hand
\begin{eqnarray*}
  &&0\leq  \frac1\delta\int_{t_1}^{t_1+\delta}\int_{x_0}^{x_1} (\partial_t \eta)^+(t,x)\phi(x) \, dx \, dt\\
  &&\leq \frac1\delta \left(\int_{t_1}^{t_1+\delta}\int_{x_0}^{x_1} ((\partial_t \eta)^+(t,x))^2\phi(x) \, dx \, dt\right)^{\frac12} \underbrace{\left(\int_{t_1}^{t_1+\delta}\int_{x_0}^{x_1} \phi(x) \, dx \, dt\right)^{\frac12}}_{\leq C\sqrt{\delta}} \\
  &&\leq \frac{C}{\sqrt{\delta}} \left(\int_{t_1}^{t_1+\delta}\int_{x_0}^{x_1} ((\partial_t \eta)^+(t,x))^2\phi(x)\, dx \, dt\right)^{\frac12}\\
  &&= C \left(\frac{1}{\delta}\int_{t_1}^{t_1+\delta}\int_{x_0}^{x_1} ((\partial_t \eta)^+(t,x))^2\phi(x)\, dx \, dt\right)^{\frac12}
\end{eqnarray*}
by the H\"{o}lder inequality, so one has
\begin{eqnarray*}
     \lim\limits_{\delta\to 0} \frac1\delta\int_{t_1}^{t_1+\delta}\int_{x_0}^{x_1}(\partial_t \eta(t,x))^+\phi(x) \, dx \, dt = 0.
\end{eqnarray*}
Finally, since
\begin{eqnarray*}
    && \frac1\delta\underbrace{\int_{x_0}^{x_1} \eta(t_1+\delta,x)\phi(x)\, dx}_{\geq 0} -\frac1\delta\underbrace{\int_{x_0}^{x_1} \eta(t_1,x)\phi(x) \, dx}_{=0}=\frac1\delta\int_{t_0}^{t_1+\delta}\int_{x_0}^{x_1}\partial_t \eta(t,x)\phi(x)\, dx \, dt \\
    &&= \frac1\delta\int_{t_1}^{t_1+\delta}\int_{x_0}^{x_1}(\partial_t \eta(t,x))^+\phi(x)\, dx \, dt - \frac1\delta\int_{t_1}^{t_1+\delta}\int_{x_0}^{x_1}(\partial_t \eta(t,x))^-\phi(x)\, dx \, dt \\
    && \leq \frac1\delta\int_{t_1}^{t_1+\delta}\int_{x_0}^{x_1}(\partial_t \eta(t,x))^+\phi(x)\, dx \, dt
\end{eqnarray*}
implies
\begin{eqnarray*}
    0\leq \frac1\delta\int_{t_1}^{t_1+\delta}\int_{x_0}^{x_1}\partial_t \eta(t,x)\phi(x)\, dx \, dt \leq \frac1\delta\int_{t_1}^{t_1+\delta}\int_{x_0}^{x_1}(\partial_t \eta(t,x))^+\phi(x)\, dx \, dt
\end{eqnarray*}
by passing to the limit $\delta\to 0$ we obtain that the property $\eqref{vanish:vel}$ holds for any non-negative $\phi\in C_c^\infty((x_0,x_1))$. Next, for any non-negative $\phi \in L^2(x_0,x_1)$, there is a sequence of non-negative $\phi_\omega\in C_c^\infty((x_0,x_1))$ such that $\phi_\omega \to \phi$ in $L^2(x_0,x_1)$. Thus, by writing
\begin{eqnarray}
    &&\frac1\delta\int_{t_1}^{t_1+\delta}\int_{x_0}^{x_1}\partial_t \eta(t,x)\phi(x) \, dx \, dt \nonumber\\
    &&= \frac1\delta\int_{t_1}^{t_1+\delta}\int_{x_0}^{x_1}\partial_t \eta(t,x)(\phi(x)-\phi_\omega(x)) \, dx \, dt+ \frac1\delta\int_{t_1}^{t_1+\delta}\int_{x_0}^{x_1}\partial_t \eta(t,x)\phi_\omega(x) \, dx \, dt, \label{sum}
\end{eqnarray}
and noticing
\begin{eqnarray*}
    &&\int_{t_1}^{t_1+\delta}\int_{x_0}^{x_1}\partial_t \eta(t,x)(\phi(x)-\phi_\omega(x)) \, dx \, dt \\
    &&\leq \delta||\partial_t \eta||_{L^\infty(0,T; L^2(x_0,x_1))} ||\phi(x)-\phi_\omega(x)||_{L^2(x_0,x_1)} \leq C\delta ||\phi(x)-\phi_\omega(x)||_{L^2(x_0,x_1)}, 
\end{eqnarray*}
one can pass to the limit $\delta\to 0$ in $\eqref{sum}$
\begin{eqnarray*}
    &&\left|\lim\limits_{\delta\to 0}\frac1\delta\int_{t_1}^{t_1+\delta}\int_{x_0}^{x_1}\partial_t \eta(t,x)\phi(x) \, dx \, dt \right|\\
    &&= \left|\lim\limits_{\delta\to 0}\frac1\delta\int_{t_1}^{t_1+\delta}\int_{x_0}^{x_1}\partial_t \eta(t,x)(\phi(x)-\phi_\omega(x)) \, dx \, dt \right|\leq C ||\phi(x)-\phi_\omega(x)||_{L^2(x_0,x_1)}.
\end{eqnarray*}
 Now, by passing to the limit $\omega\to 0$ one has that $\eqref{vanish:vel}$ holds for any non-negative $\phi\in L^2(x_0,x_1)$, which can trivially be extended to any $\phi\in L^2(x_0,x_1)$. This finishes the proof.

\section{Proof of Theorem $\ref{main3}$}

Let $\varepsilon,\delta>0$, define 
\begin{eqnarray*}
   && \varphi_\varepsilon(t,x):=\begin{cases}
        1-\frac{x-f(t)}{\varepsilon},& \quad \text{for } a< t < b \text{ and } f(t)\leq x\leq f(t)+\varepsilon,\\
        1-\frac{f(t)-x}{\varepsilon},& \quad \text{for } a< t < b \text{ and } f(t)-\varepsilon\leq x\leq f(t), \\
        0,& \quad \text{elsewhere},
    \end{cases}\\
   && \psi_\delta(t):= \begin{cases}
        \frac{t-a}{\delta},& \quad \text{for } a\leq t \leq a+\delta, \\
        \frac{b-t}{\delta},& \quad \text{for } b-\delta\leq t \leq b, \\
         1,& \quad \text{for } a+\delta < t< b-\delta, \\
           0,& \quad \text{elsewhere},
    \end{cases}
\end{eqnarray*}
and test $\eqref{momeqweak}$ with $\varphi_\varepsilon\psi_\delta$ by the density argument and pass to the limit $\delta\to 0$ to obtain
\begin{eqnarray}
    && \int_a^b\int_{f(t)-\varepsilon}^{f(t)+\varepsilon} \varphi_\varepsilon dF_{con}   \nonumber \\
    &&=-\int_a^b\int_{f(t)-\varepsilon}^{f(t)+\varepsilon} \partial_t \eta \partial_t \varphi_\varepsilon \, dx \, dt +\int_a^b\int_{f(t)-\varepsilon}^{f(t)+\varepsilon} \partial_{tx}\eta \partial_x\varphi_\varepsilon  \, dx \, dt+\int_a^b\int_{f(t)-\varepsilon}^{f(t)+\varepsilon} \partial_{x}\eta \partial_{x} \varphi_\varepsilon \, dx \, dt\nonumber\\
    &&+\int_{f(b)-\varepsilon}^{f(b)+\varepsilon} R(x) \varphi_\varepsilon(f(b),x)\, dx - \int_{f(a)-\varepsilon}^{f(a)+\varepsilon} L(x) \varphi_\varepsilon(f(a),x)\, dx, \label{F:est}
\end{eqnarray}
where last two terms are defined as the weak limits
\begin{eqnarray*}
    &&\int_{f(b)-\varepsilon}^{f(b)+\varepsilon} R(x) \varphi_\varepsilon(f(b),x))\, dx :=\lim\limits_{\delta\to 0} \frac1{\delta}\int_{a}^{a+\delta}\int_{f(b)-\varepsilon}^{f(b)+\varepsilon} \partial_t\eta(t,x) \varphi_\varepsilon(t,x)\, dx \, dt,\\
     &&\int_{f(a)-\varepsilon}^{f(a)+\varepsilon} L(x) \varphi_\varepsilon(f(a),x))\, dx :=\lim\limits_{\delta\to 0} \frac1{\delta}\int_{b-\delta}^{b}\int_{f(a)-\varepsilon}^{f(a)+\varepsilon} \partial_t\eta(t,x) \varphi_\varepsilon(t,x)\, dx \, dt,
\end{eqnarray*}
which exist by boundedness of $\partial_t \eta$ in $L^\infty(0,T; L^2(0,l))$. First, we calculate
\begin{eqnarray*}
    &&\int_a^b\int_{f(t)-\varepsilon}^{f(t)+\varepsilon}\partial_x \eta \partial_x\varphi_\varepsilon \, dx \, dt = \frac1\varepsilon\int_a^b \int_{f(t)-\varepsilon}^{f(t)} \partial_x \eta \, dx \, dt - \frac1\varepsilon\int_a^b \int_{f(t)}^{f(t)+\varepsilon} \partial_x \eta \, dx \, dt,\\
     &&\int_a^b\int_{f(t)-\varepsilon}^{f(t)+\varepsilon}\partial_{tx} \eta \partial_x\varphi_\varepsilon \, dx \, dt = \frac1\varepsilon\int_a^b \int_{f(t)-\varepsilon}^{f(t)} \partial_{tx} \eta \, dx \, dt - \frac1\varepsilon\int_a^b \int_{f(t)}^{f(t)+\varepsilon} \partial_{tx} \eta \, dx \, dt.
\end{eqnarray*}
Next
\begin{eqnarray*}
    &&\int_a^b\int_{f(t)-\varepsilon}^{f(t)+\varepsilon} \partial_t \eta \partial_t \varphi_\varepsilon \, dx \, dt =  -\frac{1}{\varepsilon}\int_a^b\int_{f(t)-\varepsilon}^{f(t)} \partial_t \eta \partial_t f \, dx \, dt + \frac{1}{\varepsilon}\int_a^b\int_{f(t)}^{f(t)+\varepsilon} \partial_t \eta \partial_t f \, dx \, dt \\
    &&= -\int_a^b \partial_t \eta(t,f(t)-\theta_1\varepsilon) \partial_t f(t) \, dt+ \int_a^b \partial_t \eta(t,f(t)+\theta_2\varepsilon) \partial_t f(t) \, dt \\
    &&=  \int_a^b \int_{f(t)-\theta_1\varepsilon}^{f(t)+\theta_2\varepsilon} \partial_{tx} \eta \partial_t f\, dx \, dt \\
    &&\leq C\sqrt{\varepsilon} || \partial_{tx}\eta ||_{L^2((0,T)\times(0,l))} ||\partial_t f||_{L^2(t_0-r,t_0+r)} \leq C\sqrt{\varepsilon}, 
\end{eqnarray*}
by the mean-value Theorem, where $\theta_1,\theta_2 \in [0,1]$. Finally 
\begin{eqnarray*}
    &&\int_{f(b)-\varepsilon}^{f(b)+\varepsilon} R(x) \varphi_\varepsilon(f(b),x)\, dx - \int_{f(a)-\varepsilon}^{f(a)+\varepsilon} L(x) \varphi_\varepsilon(f(a),x)\, dx\\
    &&\leq ||\partial_t \eta||_{L^\infty(0,T; L^2(0,l))}\left(\left(\int_{f(a)-\varepsilon}^{f(a)+\varepsilon} |\varphi_\varepsilon((a),x)\, dx|^2   \right)^{\frac12}+\left(\int_{f(b)-\varepsilon}^{f(b)+\varepsilon} |\varphi_\varepsilon(f(b),x)\, dx|^2   \right)^{\frac12}\right) \\
    &&\leq C\sqrt{\varepsilon}.
\end{eqnarray*}
Passing to the limit $\varepsilon\to 0$ in $\eqref{F:est}$ gives us $\eqref{iden:force}$. \\

Next, assume that $\{(t,f(t)): t\in [a,b]\}\subset \partial\{\eta=0\}\cap \text{int}\{\partial_t \eta\geq 0\}$. The sketch of the proof is as follows. First, we will show that $f$ has to be monotone on $[a,b]$, which will ensure that $\partial_t \eta$ vanishes at $\{(t,f(t)): t\in [a,b]\}$ by Theorem $\ref{main2}$, so $\partial_t\eta$ and trivially $\eta$ both reach its minimum and therefore satisfy $\partial_{xx}\eta,\partial_{txx} \eta\geq 0$ on $\{(t,f(t)): t\in [a,b]\}$. This means that the jump of $-\partial_{x}\eta-\partial_{tx}\eta$ in $x$-direction has to be non-positive, while on the other hand $\eqref{iden:force}$ implies that this jump is equal to $F_{con}\geq 0$, so we conclude $F_{con}=0$. \\

Let us therefore first show that $f$ is locally monotonous. Fix $t_0\in(a,b)$. Since $\{(t,f(t)): a\leq t \leq b\} \subset  \text{int}\{\partial_t \eta \geq 0\}$, there exists a rectangle
\begin{eqnarray*}
    Q:=[t_0-\theta,t_0+\theta]\times \left[\min\limits_{t\in [t_0-\theta,t_0+\theta]}f(t),\max\limits_{t\in [t_0-\theta,t_0+\theta]}f(t)\right]
\end{eqnarray*}
for some $\theta>0$ such that $Q\subset \text{int}\{\partial_t \eta \geq 0\}$. Assume that there are two points $\theta_1,\theta_2 \in [t_0-\theta,t_0+\theta] $ such that $f(\theta_1)=f(\theta_2)$. We will show that this implies $f=\text{const}$ on $[\theta_1,\theta_2]$. Assume the opposite, then w.l.o.g. there exists $\theta'\in (\theta_1,\theta_2)$ such that the minimum of $f$ on $[\theta_1,\theta_2]$ is achieved in $f(\theta')$. First, since $Q\subset \text{int}\{\partial_t \eta \geq 0\}$, this means that
\begin{eqnarray*}
    (t,f(t))\in Q \implies \eta = 0 \text{ on }[t_0-\theta,t]\times \{f(t)\},
\end{eqnarray*}
so 
\begin{eqnarray*}
    \bigcup_{t\in [\theta',\theta_2]}[t_0-\theta,t]\times \{f(t)\} \subset \{\eta = 0\}.
\end{eqnarray*}
However, by continuity of $f$, there exist $\theta'',\theta'''$ such that
\begin{eqnarray*}
    \theta_1 < \theta''< \theta' < \theta'''<\theta_2, \qquad f(\theta_1)>f(\theta'')>f(\theta')<f(\theta''')<f(\theta_2)=f(\theta_1),
\end{eqnarray*}
which implies that 
\begin{eqnarray*}
    (\theta'',f(\theta''))\in ( \theta_1, \theta')\times (f(\theta'),f(\theta_2))\subset\text{int}\left(\bigcup_{t\in [\theta',\theta_2]}[t_0-\theta,t]\times \{f(t)\} \right) \subset \text{int}\{\eta = 0\},
\end{eqnarray*}
so $(\theta'',f(\theta''))$ cannot belong to $\partial\{\eta = 0\}$, and that is a contradiction. Thus, this means that for every point $t_0 \in (a,b)$, there is a neighborhood $[t_0-\theta,t_0+\theta]$ such that $f$ is monotone on $[t_0-\theta,t_0+\theta]$, so $f$ is locally monotone and therefore monotone on $(a,b)$. Therefore, by Theorem $\ref{main2}$, $\partial_t \eta$ vanishes on the graph of $f$ in the sense of $\eqref{par:int:1}$ and $\eqref{par:int:2}$. Taking into consideration that $\partial_t \eta (t,f(t))\sqrt{1+f'(t)^2} \in L^1(a,b)$ since $f\in H^1(a,b)$ by assumption, this implies that $\partial_t \eta(t,f(t))\sqrt{1+f'(t)^2}=0$ a.e. on $(a,b)$ so
\begin{eqnarray*}
   &&\int_a^b\int_{f(t)-\varepsilon}^{f(t)+\varepsilon}\partial_{tx} \eta \partial_x\varphi_\varepsilon\, dx \, dt = \frac1\varepsilon\int_a^b \int_{f(t)-\varepsilon}^{f(t)} \partial_{tx} \eta \, dx \, dt- \frac1\varepsilon\int_a^b \int_{f(t)}^{f(t)+\varepsilon} \partial_{tx} \eta\, dx \, dt \\
    && =-\underbrace{\frac1\varepsilon\int_{a}^b\partial_t\eta(t,f(t)-\varepsilon)\sqrt{1+f'(t)^2}\, dt}_{\geq 0}+\underbrace{\frac1\varepsilon\int_{a}^b\partial_t\eta(t,f(t)) \sqrt{1+f'(t)^2}\, dt}_{=0}\\
    &&\quad +\underbrace{\frac1\varepsilon\int_{a}^b\partial_t\eta(t,f(t))\sqrt{1+f'(t)^2}\, dt}_{=0}-\underbrace{\frac1\varepsilon\int_{a}^b\partial_t\eta(t,f(t)+\varepsilon)\sqrt{1+f'(t)^2}\, dt}_{\geq 0} \leq 0.
\end{eqnarray*}
Here, $\varepsilon$ is chosen to satisfy $0<\varepsilon<\varepsilon_0$, where $\varepsilon_0>0$ is small enough so that the graphs $\{(t,f(t)+\varepsilon): a\leq t \leq b\} \subset \text{int}\{\partial_t \eta \geq 0\}$ and $\{(t,f(t)-\varepsilon): a\leq t \leq b\} \subset \text{int}\{\partial_t \eta \geq 0\}$, for any $0<\varepsilon<\varepsilon_0$. Note that such $\varepsilon_0$ exists since the set $\{(t,f(t)+\varepsilon): a\leq t \leq b\}$ is closed and thus has a positive distance from the boundary of an open set $\partial \{\partial_t \eta \geq 0\}$.
Now, since 
\begin{eqnarray*}
      &&\int_a^b\int_{f(t)-\varepsilon}^{f(t)+\varepsilon}\partial_x \eta \partial_x\varphi_\varepsilon \, dx \, dt = \frac1\varepsilon\int_a^b \int_{f(t)-\varepsilon}^{f(t)} \partial_x \eta \, dx \, dt- \frac1\varepsilon\int_a^b \int_{f(t)}^{f(t)+\varepsilon} \partial_x \eta\, dx \, dt \\
    && =-\underbrace{\frac1\varepsilon\int_{a}^b\eta(t,f(t)-\varepsilon)\sqrt{1+f'(t)^2}\, dt}_{\geq 0}+\underbrace{\frac1\varepsilon\int_{a}^b\eta(t,f(t))\sqrt{1+f'(t)^2}\, dt}_{= 0}\\
    &&\quad +\underbrace{\frac1\varepsilon\int_{a}^b\eta(t,f(t))\sqrt{1+f'(t)^2}\, dt}_{= 0}-\underbrace{\frac1\varepsilon\int_{a}^b\eta(t,f(t)+\varepsilon)\sqrt{1+f'(t)^2}\, dt}_{\geq 0} \leq 0,
\end{eqnarray*}
by $\eqref{iden:force}$ and previous two inequalities we conclude that 
\begin{eqnarray*}
   0\leq F_{con}\Big(\{(t,g(t)): a\leq t \leq b\}\cap \text{int}\{\partial_t \eta \geq 0\}\Big) \leq 0,
\end{eqnarray*}
so $\eqref{iden:forcea3}$ follows.

\bigskip

In order to prove $\eqref{iden:force2}$, assume that $\{(t,f(t)): a\leq t \leq b\} \subset  \text{int}\{\partial_t \eta \geq 0\}$. We choose $\varphi(t,x) = \psi_\delta(t)\phi(x)$ in $\eqref{momeqweak}$, where $\phi\in C_c^\infty(x_0,x_1)$, and
\begin{eqnarray*}
      \psi_\delta(t)=\begin{cases}
        0,& t\leq t_1-\delta, \\
        -\frac1\delta(t-(t_1-\delta)),& t_1-\delta\leq t \leq t_1, \\
        \frac1\delta(t-(t_1+\delta)), & t_1\leq t \leq t_1+\delta,\\
         0,& t> t_1+\delta.
    \end{cases}
\end{eqnarray*}
We then pass to the limit $\delta \to 0$, and by a density argument, extend the test functions to $L^2(x_0,x_1)$, ensuring validity of $\ref{iden:force2}$ almost everywhere on $(x_0,x_1)$. \\

Now, if $\{t_1\}\times (x_0,x_1)\subseteq \text{int}\{\partial_t \eta \geq 0\}$, there is a $t_0<t_1$ such that $(t_0,t_1)\times (x_0,x_1)\subseteq \text{int}\{\partial_t \eta \geq 0\}$ which implies that $\eta=0$ on $(t_0,t_1)\times (x_0,x_1)$, so by using $\eqref{vanish:vel}$ we conclude $\eqref{iden:force2a3}$. This concludes the proof.

\section{Numerical examples}\label{numerics}
In the final section, we present two numerical examples to illustrate the theoretical results. The equations are discretized as follows:
\begin{eqnarray*}
    \begin{cases}
    &\frac{\eta_{j}^{i+1}-2\eta_j^{i}+\eta_j^{i-1}}{(\Delta t)^2}  - \frac\alpha{\Delta t}\left(\frac{\eta_{j+1}^{i+1}-2\eta_{j}^{i+1}+\eta_{j-1}^{i+1}}{(\Delta x)^2} - \frac{\eta_{j+1}^{i}-2\eta_{j}^{i}+\eta_{j-1}^{i}}{(\Delta x)^2}\right) - \frac{\eta_{j+1}^{i+1}-2\eta_{j}^{i+1}+\eta_{j-1}^{i+1}}{(\Delta x)^2} = \frac1\varepsilon \chi_{\{\eta_{j}^i<0\}} \left(\frac{\eta_{j}^{i}-\eta_j^{i-1}}{\Delta t}\right)^-,\\ \\
    &\eta^{i}_0 =  \eta^{i}_{N} = 0, \quad i \geq 1, 
    \end{cases}
\end{eqnarray*}
where $0\leq i \leq \frac{l}{\Delta x}=: N \in \mathbb{N}$, $1\leq j \leq \frac{T}{\Delta t}=: M \in \mathbb{N}$ and $\eta^i = [ \eta_0^i~,~ \eta_1^i~,~ \dots ~, ~\eta_N^i]^T$, with $\alpha>0$ being the viscoelasticity coefficient. Note that the penalization force is taken explicitly from the previous time step.

All numerical computations were performed using MATLAB. We tested the numerical convergence with respect to both the time and space discretization parameters, as well as the penalization parameter $\varepsilon$.

\subsection{Example 1}
For this example, we fix $l=1,~ T=0.3,~ \Delta t = \Delta x = 1/5000,~ \alpha=0.01,~ \varepsilon=0.0005$, while the initial data is
\begin{eqnarray*}
    \eta^{0} = 1+ \frac12 \sin^2(10 \pi x), \quad v^0 = -50.
\end{eqnarray*}

\begin{figure}[H]
     \centering
     \begin{subfigure}[b]{0.49\textwidth}
         \centering
         \includegraphics[width=\textwidth]{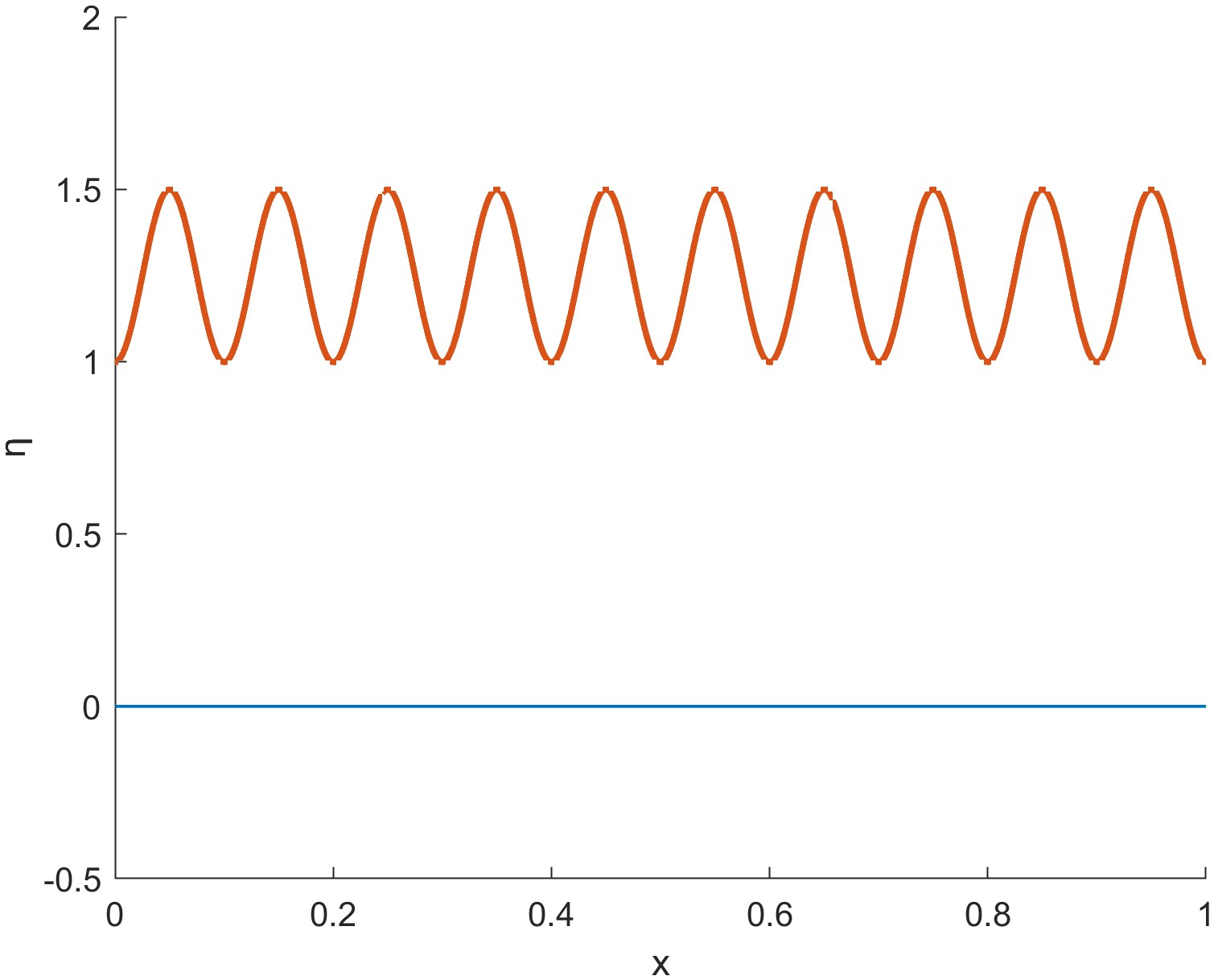}
         \caption{$t=0$}
         \label{fig1:a}
     \end{subfigure}
     \hfill
     \begin{subfigure}[b]{0.49\textwidth}
         \centering
         \includegraphics[width=\textwidth]{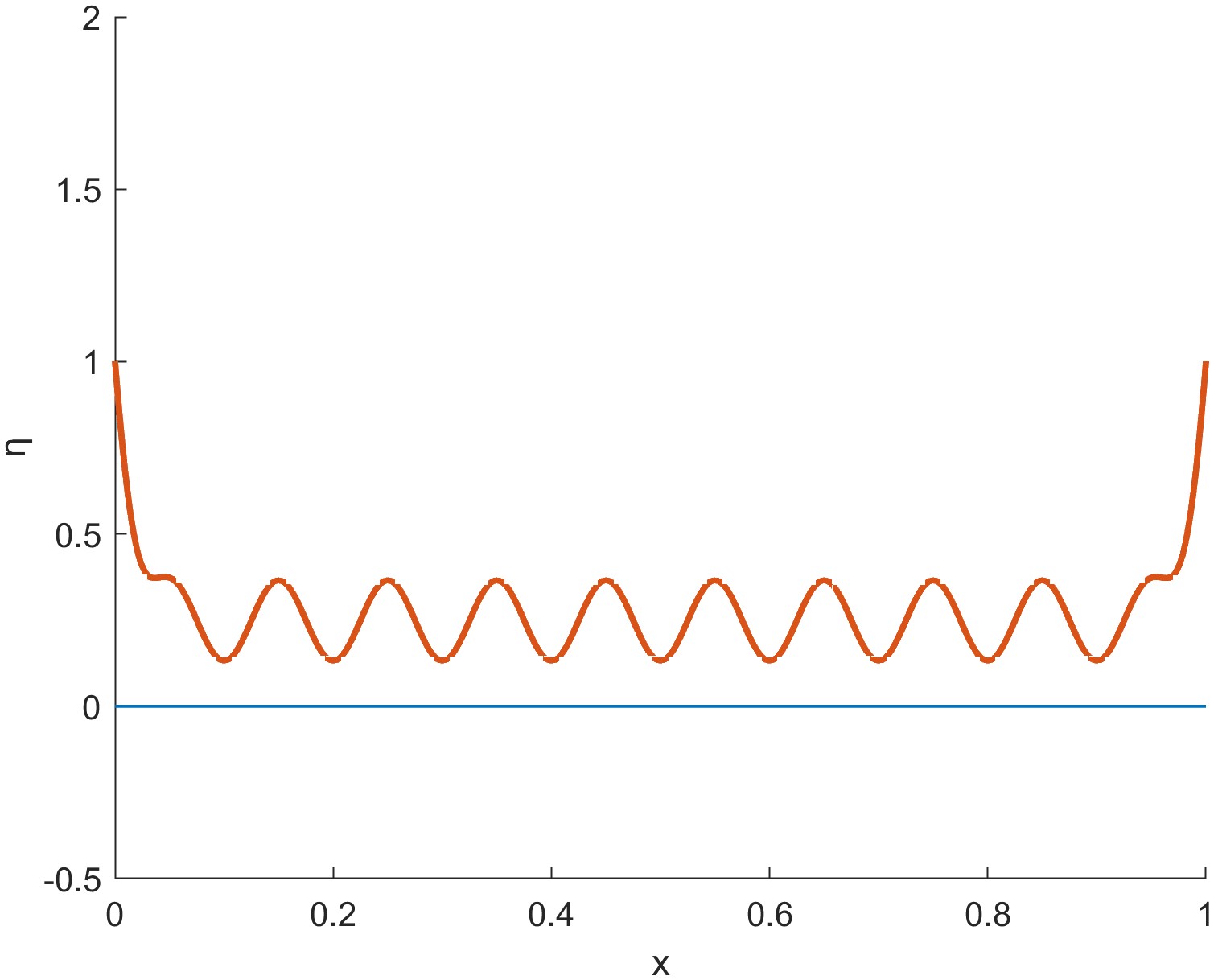}
         \caption{$t=0.02$}
         \label{12}
     \end{subfigure} \\
          \centering
     \begin{subfigure}[b]{0.49\textwidth}
         \centering
         \includegraphics[width=\textwidth]{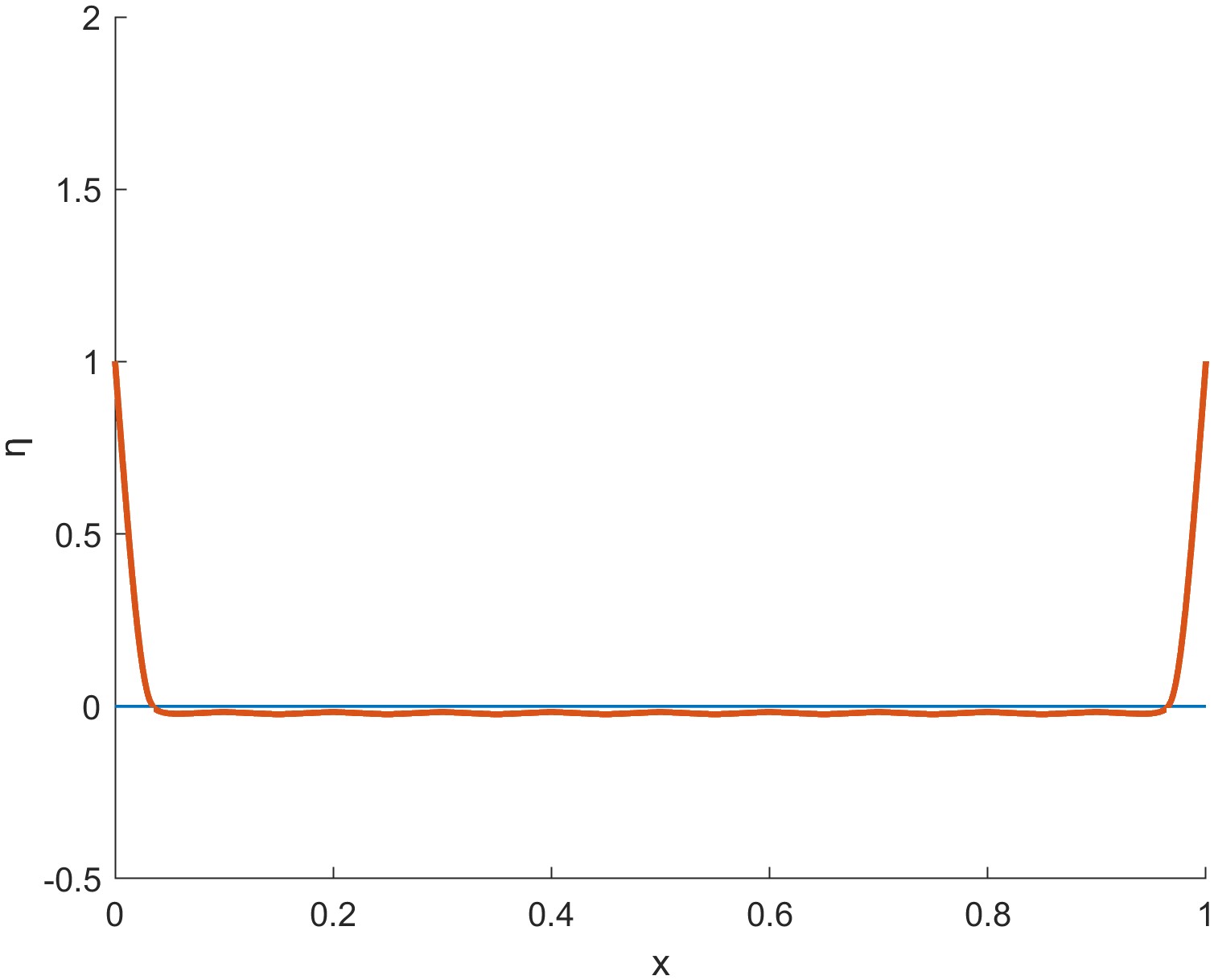}
         \caption{$t=0.04$}
         \label{13}
     \end{subfigure}
     \hfill
     \begin{subfigure}[b]{0.49\textwidth}
         \centering
         \includegraphics[width=\textwidth]{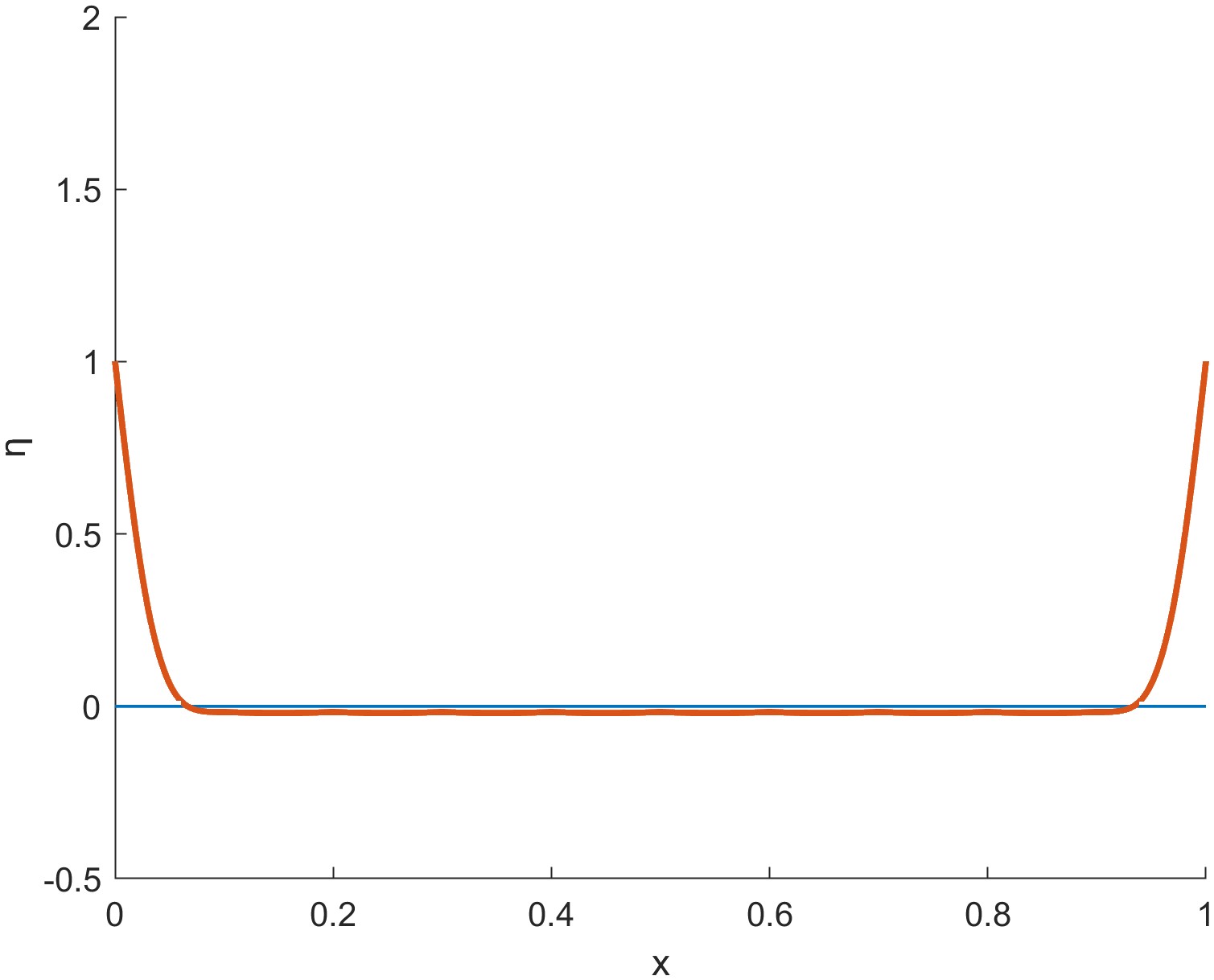}
         \caption{$t=0.06$}
         \label{14}
     \end{subfigure}
   \begin{subfigure}[b]{0.49\textwidth}
         \centering
\includegraphics[width=\textwidth]{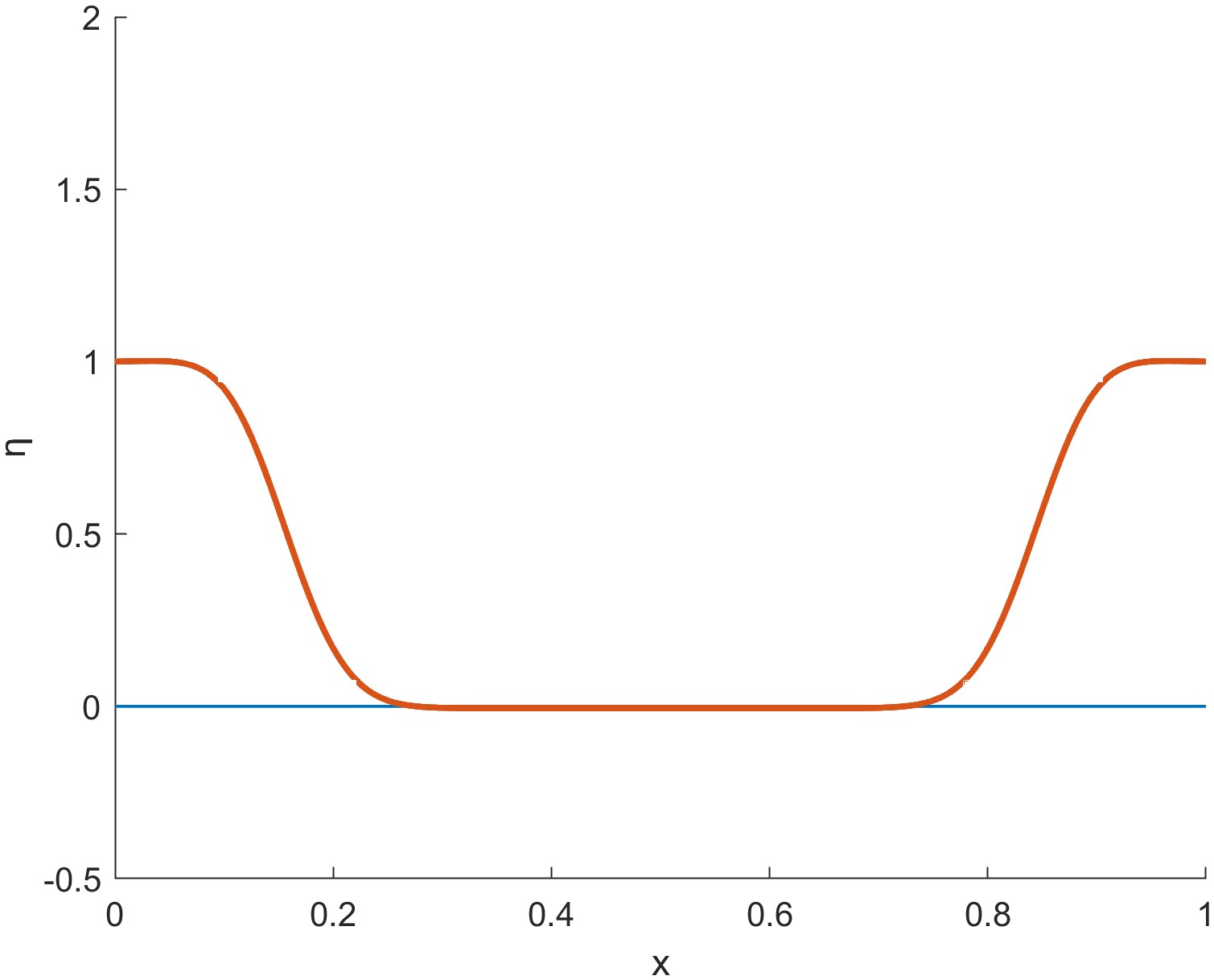}
         \caption{$t=0.2$}
         \label{15}
     \end{subfigure}
     \hfill
     \begin{subfigure}[b]{0.49\textwidth}
         \centering
         \includegraphics[width=\textwidth]{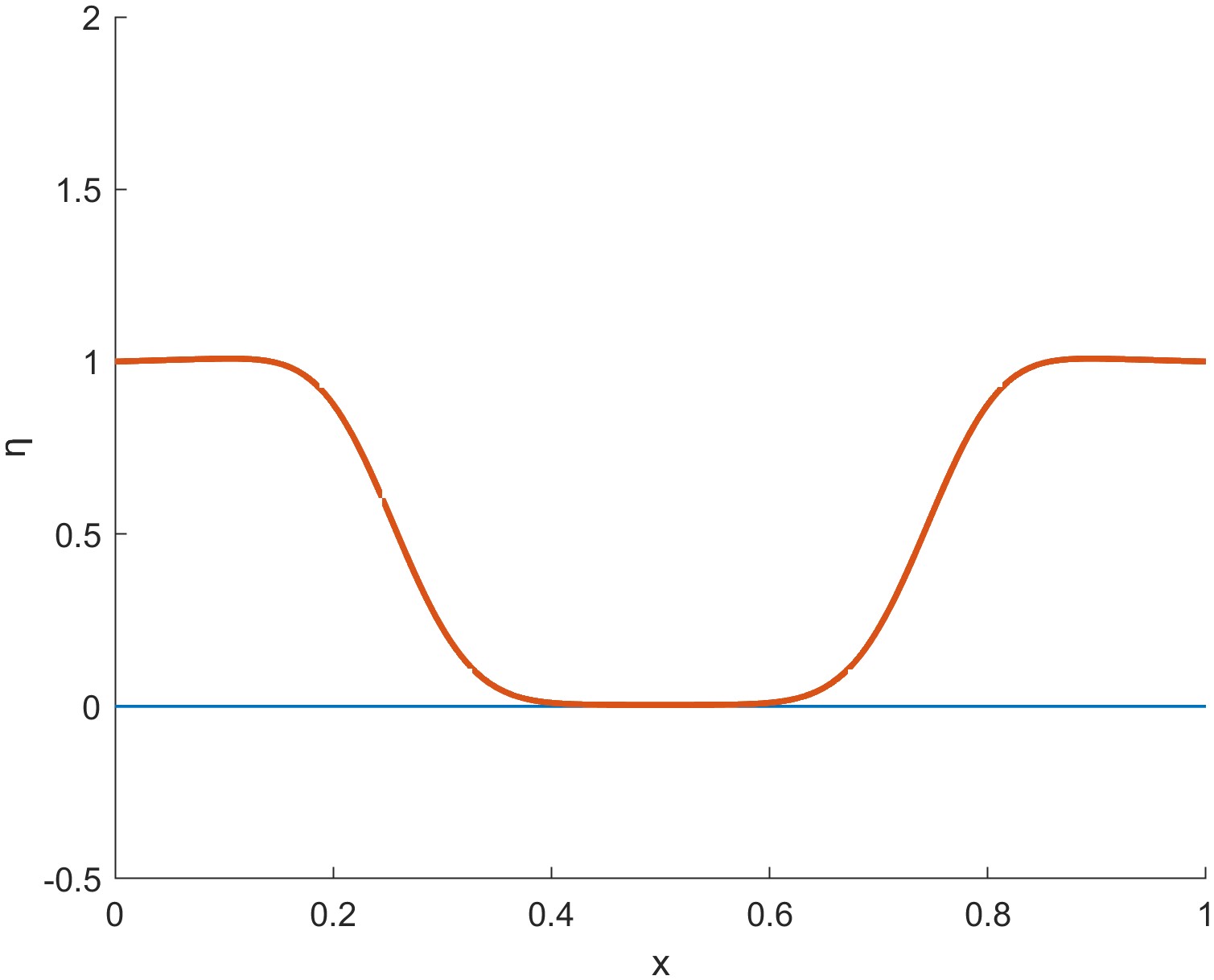}
         \caption{$t=0.3$}
         \label{16}
     \end{subfigure}
     
\caption{Solution of example 1 at different times.}
\label{momenti1}
     
\end{figure}
First, in Figures~\ref{momenti1}(\subref*{fig1:a}) -~\ref{momenti1}(\subref*{16}), we represent the solution at different times, then in figure \ref{ex1} we also display details of the contact set and the velocity field. In this example, we illustrate how contact forms and evolves within our model, and how the formulation allows for detachment from the obstacle. The simulation also highlights how contact leads to energy dissipation: oscillations are visibly damped once contact occurs. We prescribe symmetric, oscillatory initial data for the displacement and impose an initial downward velocity. As a result, contact forms and quickly spreads across almost the entire length of the string. Within the contact region, oscillations are almost entirely suppressed due to localized damping. At later times, the boundary conditions begin to "pull back" the string, reducing the size of the contact set. Owing to the symmetry of the initial data, the contact set also remains symmetric throughout the evolution.
\begin{figure}[H]
    \begin{subfigure}[b]{0.49\textwidth}
    \centering
    \includegraphics[width=0.7\linewidth]{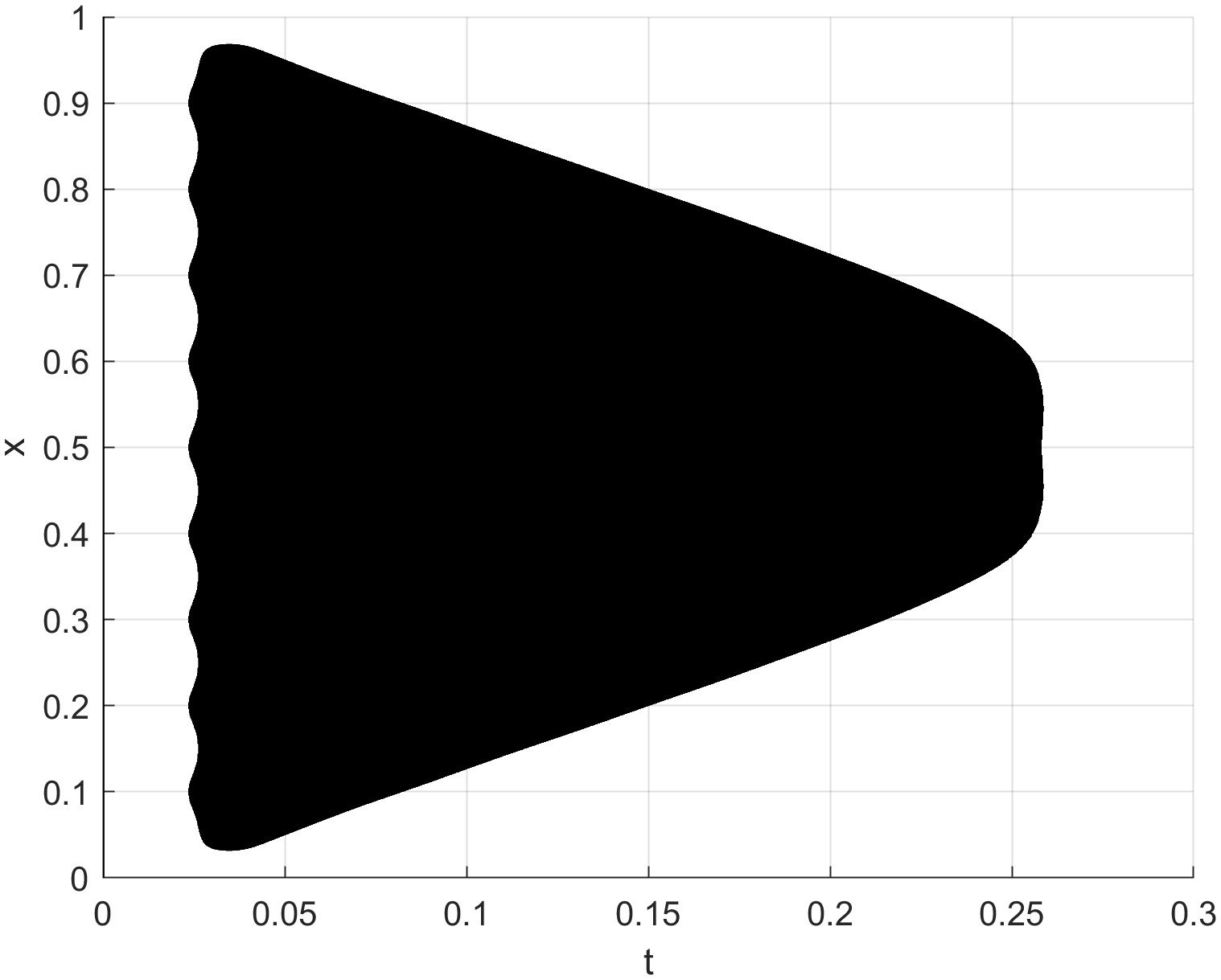}
\end{subfigure}
 \hfill
\begin{subfigure}[b]{0.49\textwidth}
\centering
    \includegraphics[width=0.7\linewidth]{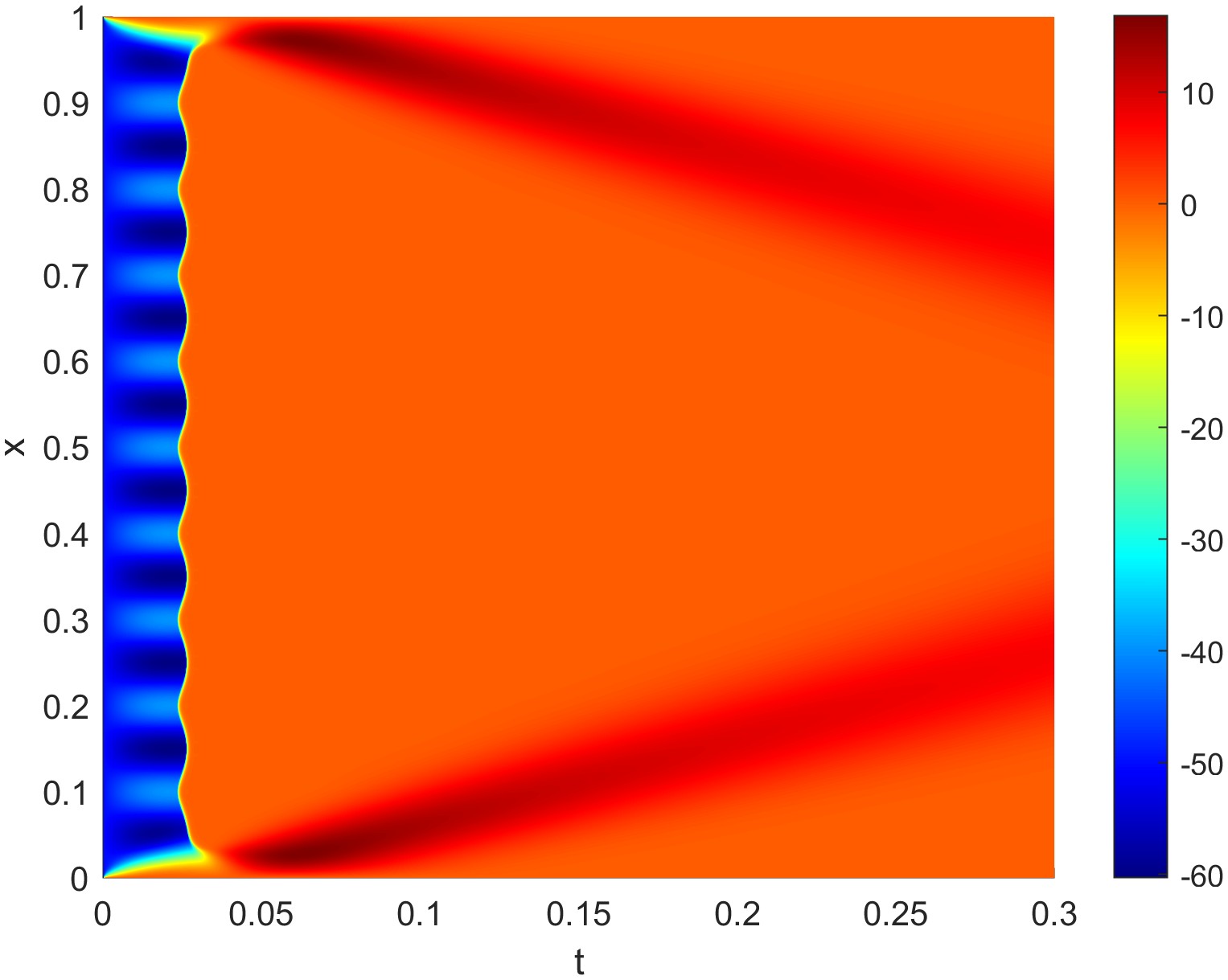}
    \end{subfigure}
  
    \caption{Contact set (left) and velocity field (right) for example 1.} 
      \label{ex1}
\end{figure}

\subsection{Example 2}
For this example, we fix $l=1,~ T=0.5,~ \Delta t = \Delta x = 1/5000,~ \alpha=0.01,~ \varepsilon=0.0005$, the initial data is
\begin{eqnarray*}
  \eta^0 = \begin{cases}
       x, \quad &0\leq x<0.2,\\
       \sin(\pi(x-0.2)/0.3 ), \quad  &0.2 \leq x < 0.8, \\
       2-x, \quad &0.8\leq x<1, 
   \end{cases}  \qquad v^0 = \begin{cases}
       -50, \quad &0\leq x<0.6,\\
       -0.5, \quad  &0.6 \leq x \leq 1.
   \end{cases}
\end{eqnarray*}
\begin{figure}[H]
     \centering
     \begin{subfigure}[b]{0.49\textwidth}
         \centering
         \includegraphics[width=\textwidth]{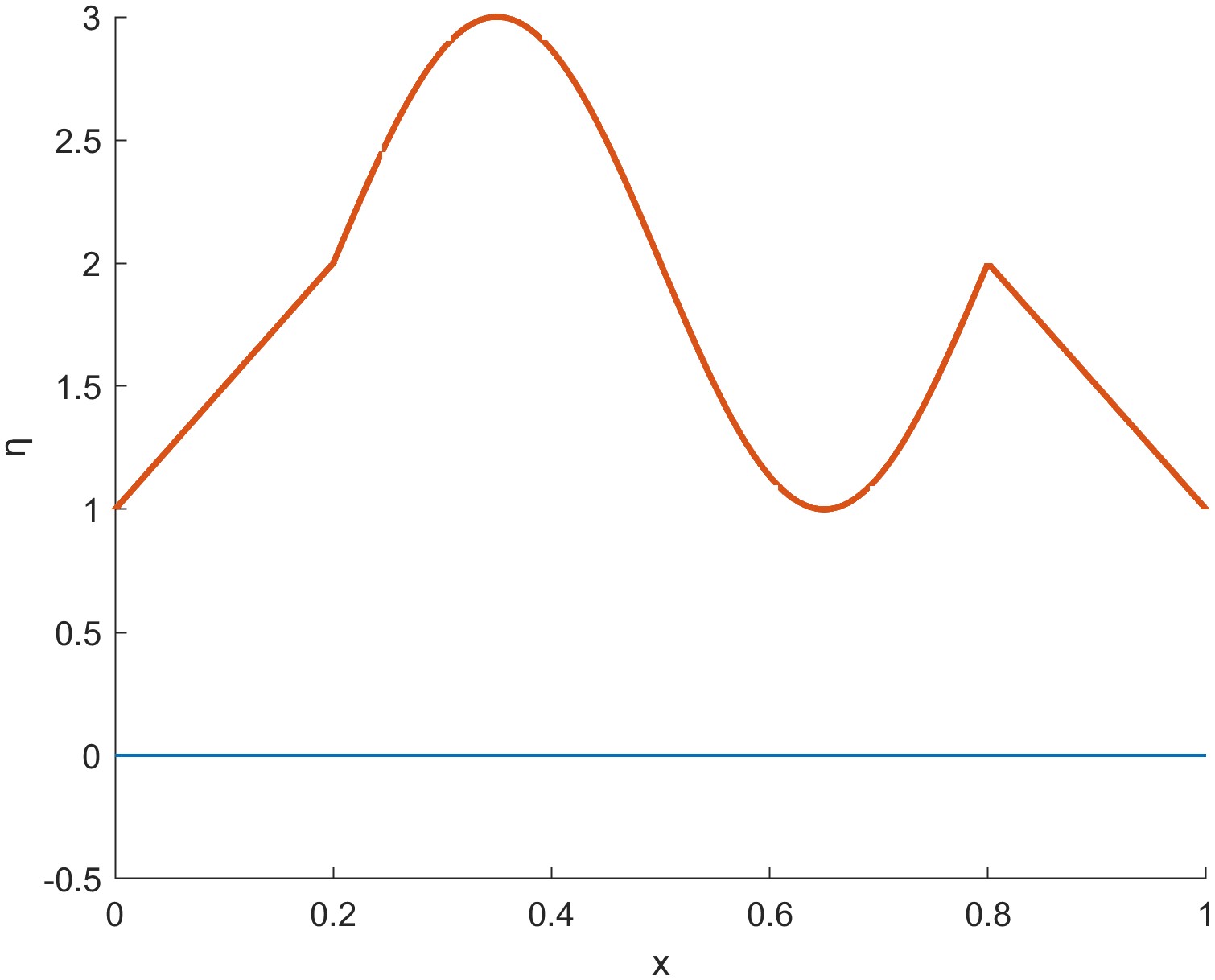}
         \caption{$t=0$}
         \label{21}
     \end{subfigure}
     \hfill
     \begin{subfigure}[b]{0.49\textwidth}
         \centering
         \includegraphics[width=\textwidth]{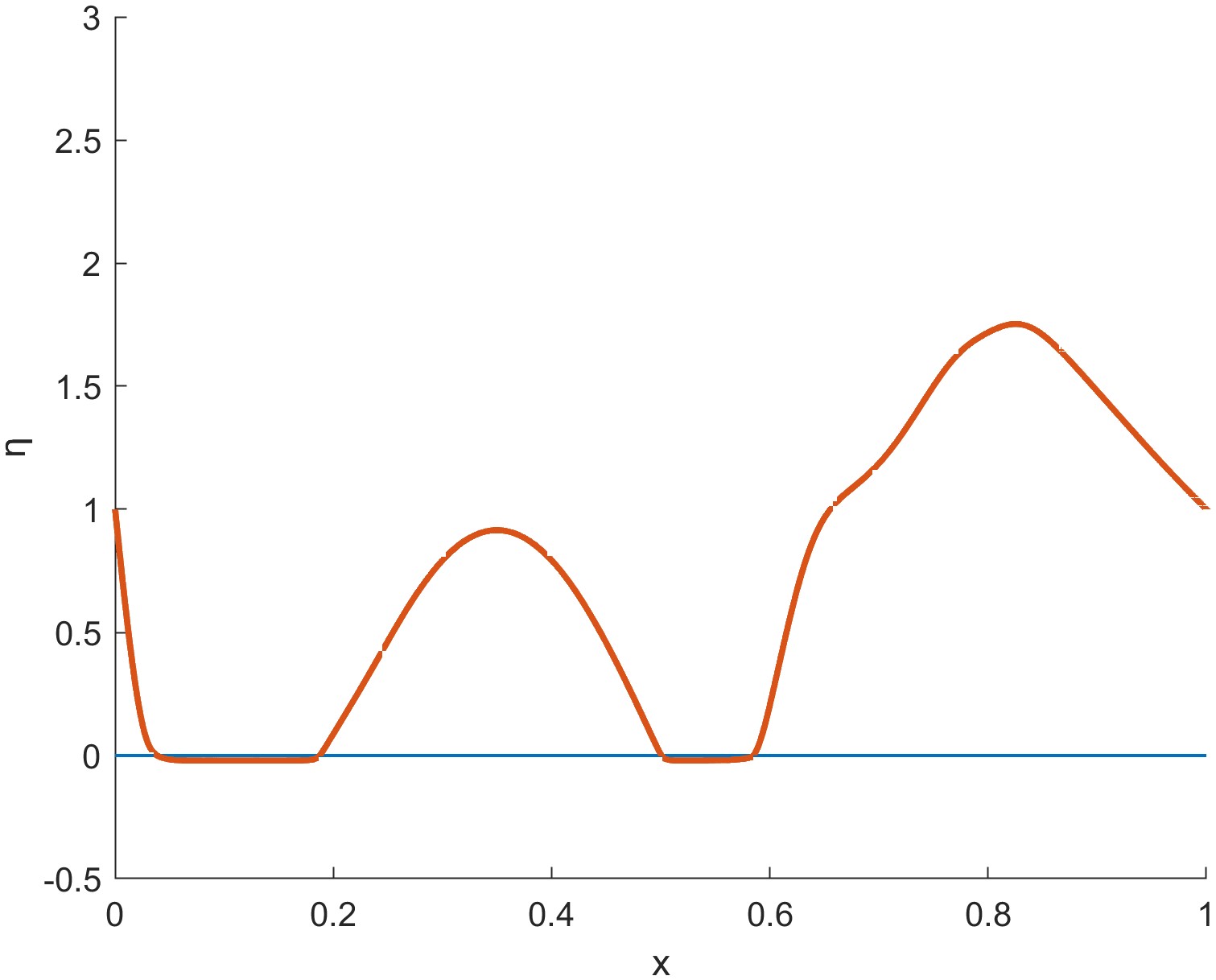}
         \caption{$t=0.04$}
         \label{22}
     \end{subfigure} \\
          \centering
     \begin{subfigure}[b]{0.49\textwidth}
         \centering
         \includegraphics[width=\textwidth]{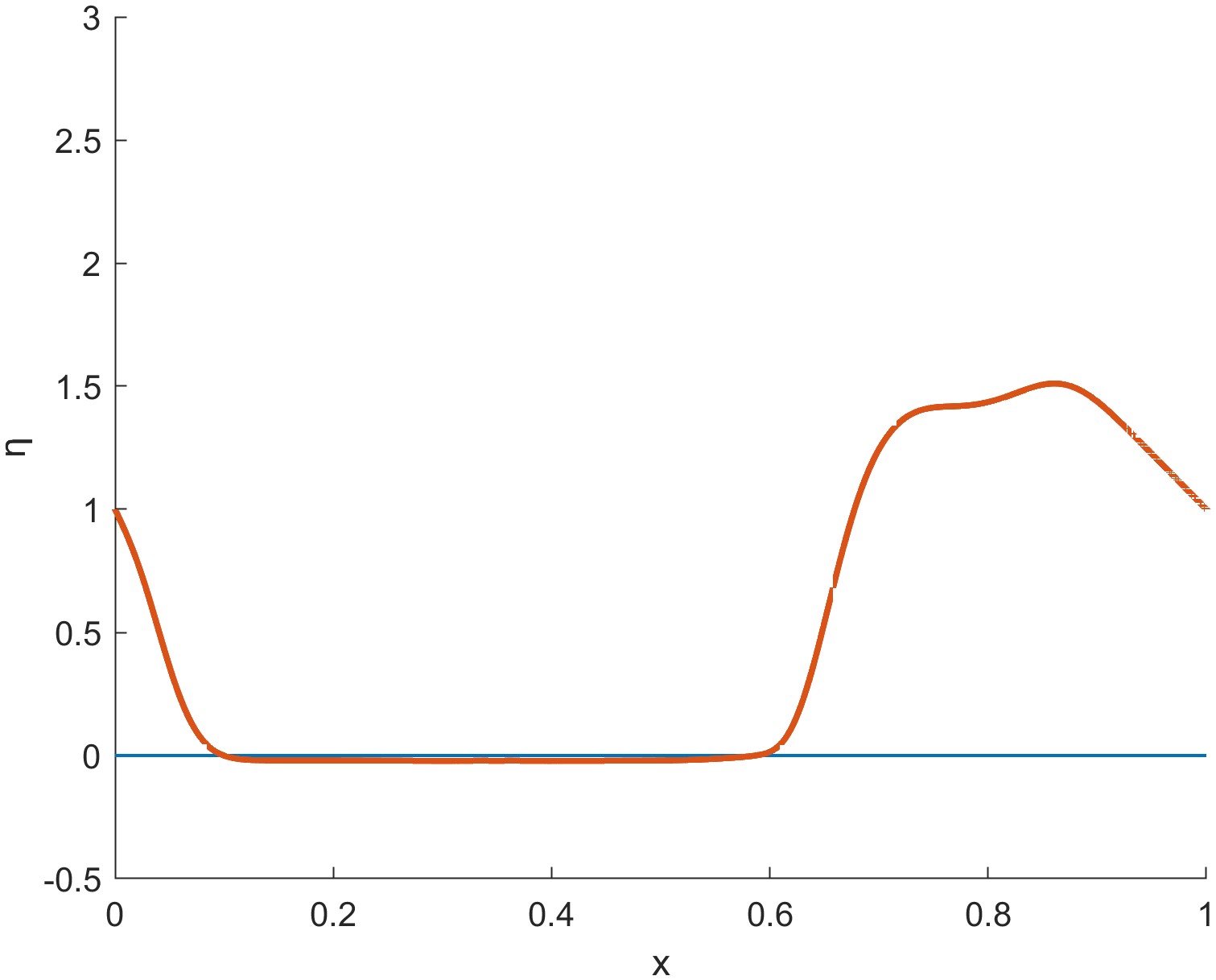}
         \caption{$t=0.08$}
         \label{23}
     \end{subfigure}
     \hfill
     \begin{subfigure}[b]{0.49\textwidth}
         \centering
         \includegraphics[width=\textwidth]{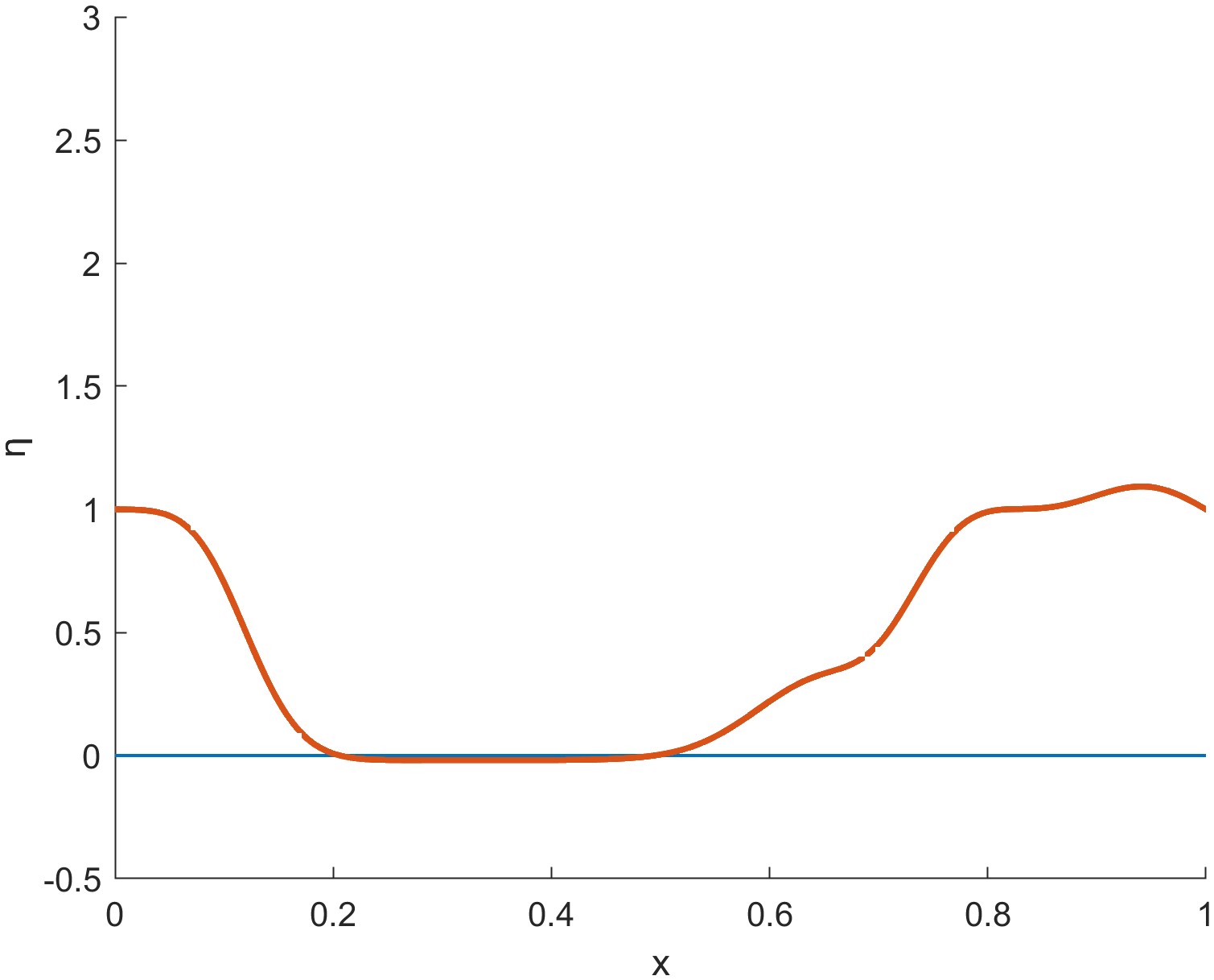}
         \caption{$t=0.16$}
         \label{24}
     \end{subfigure}
   \begin{subfigure}[b]{0.49\textwidth}
         \centering
\includegraphics[width=\textwidth]{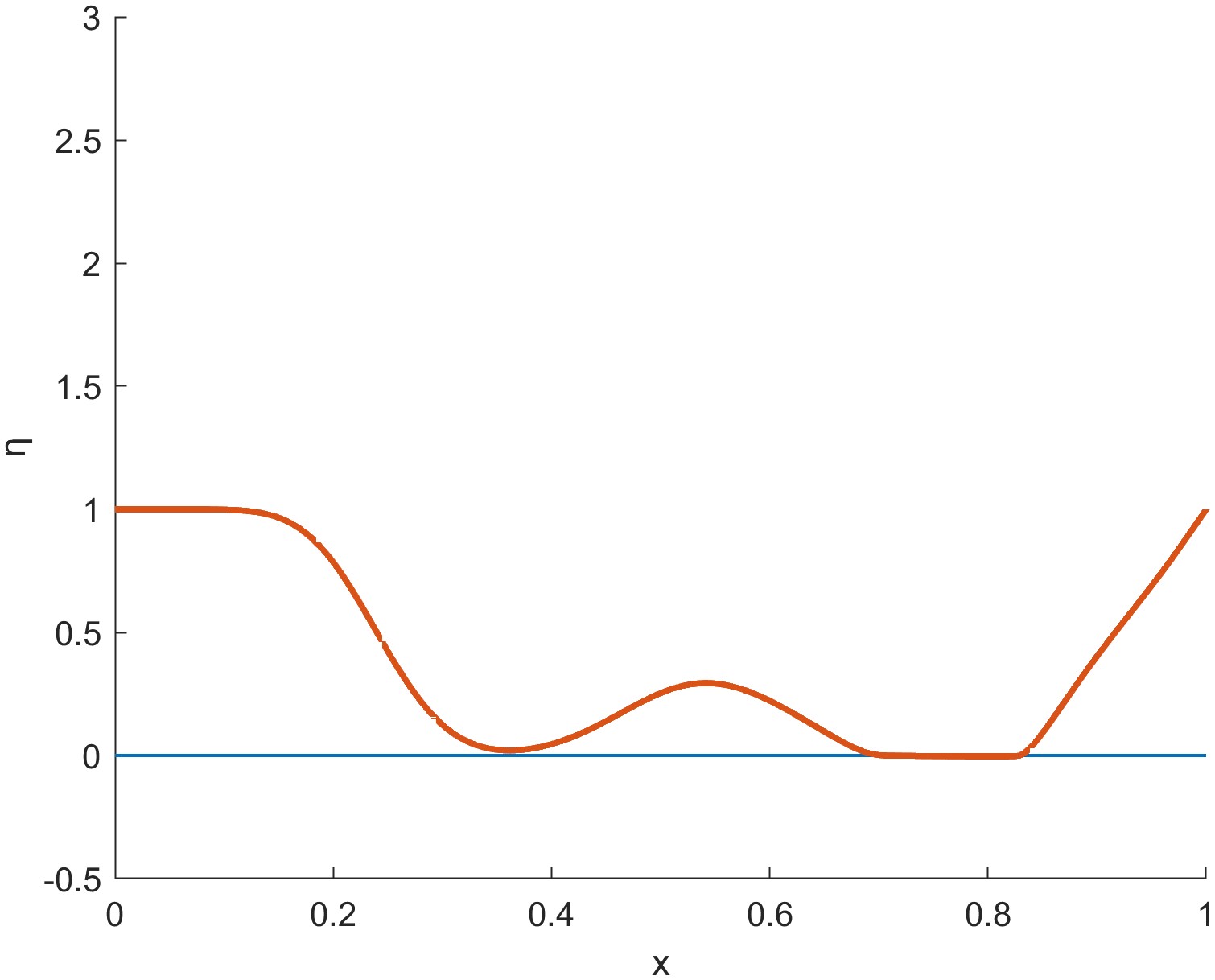}
         \caption{$t=0.28$}
         \label{25}
     \end{subfigure}
     \hfill
     \begin{subfigure}[b]{0.49\textwidth}
         \centering
         \includegraphics[width=\textwidth]{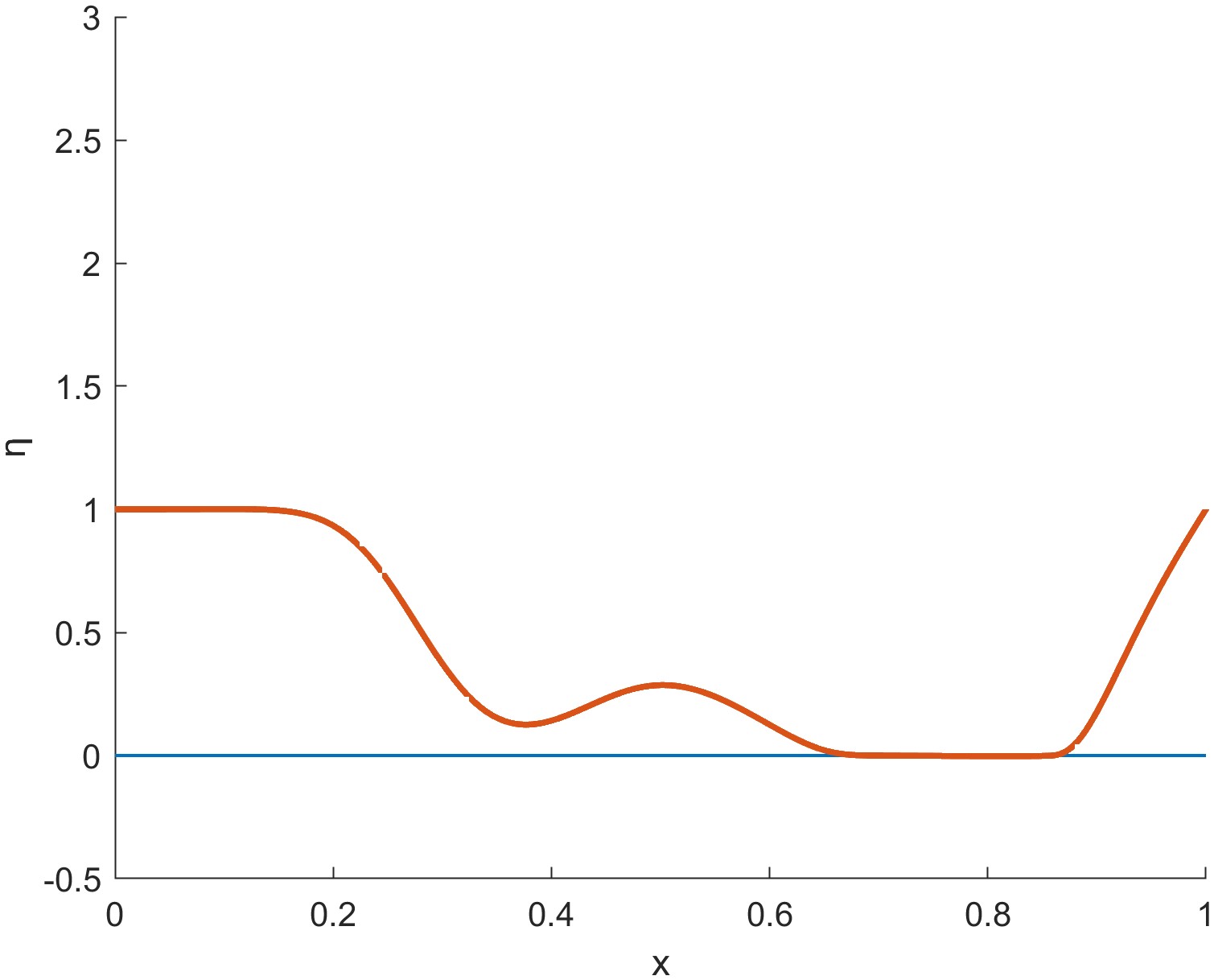}
         \caption{$t=0.32$}
         \label{26}
     \end{subfigure}
     
\caption{Solution of example 1 at different times.}
\label{momenti2}
     
\end{figure}
In Figures~\ref{momenti2}(\subref*{21}) -~\ref{momenti2}(\subref*{26}), we represent the solution at different times, then in Figure \ref{ex2} we also display details of the contact set and the velocity field. The setting of this example is similar to the first, but the initial data are non-oscillatory. This setup demonstrates that the contact set can consist of multiple disconnected components. Due to the shape of the initial displacement and the imposed downward velocity, contact forms in separate regions along the string rather than as a single connected segment. This highlights that, even in the absence of oscillations, the model naturally accommodates complex contact geometries arising from localized interactions with the obstacle.
\begin{figure}[H]
\begin{subfigure}[b]{0.49\textwidth}
        \centering
    \includegraphics[width=0.7\linewidth]{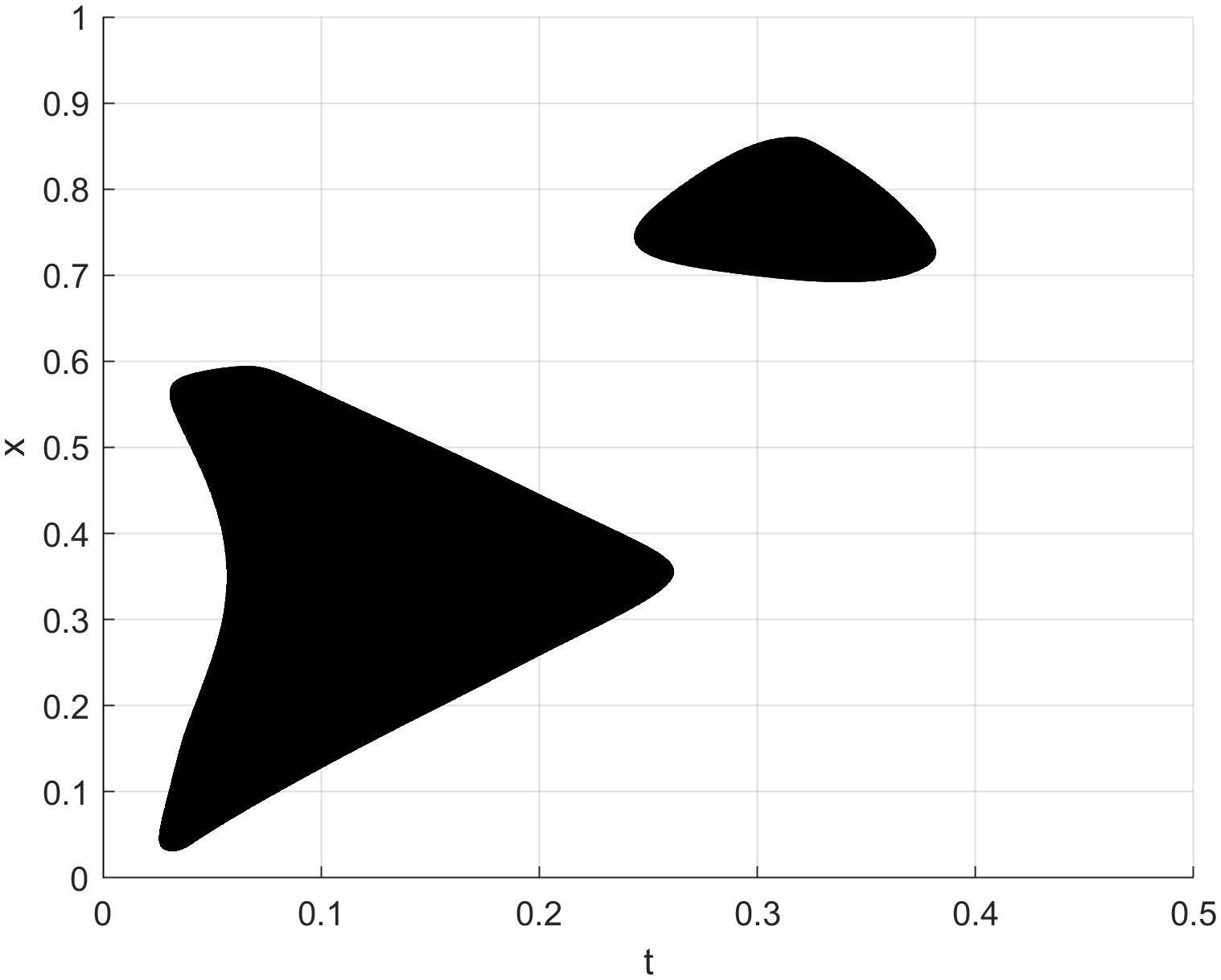}
\end{subfigure}
\hfill
\begin{subfigure}[b]{0.49\textwidth}
    \centering
    \includegraphics[width=0.7\linewidth]{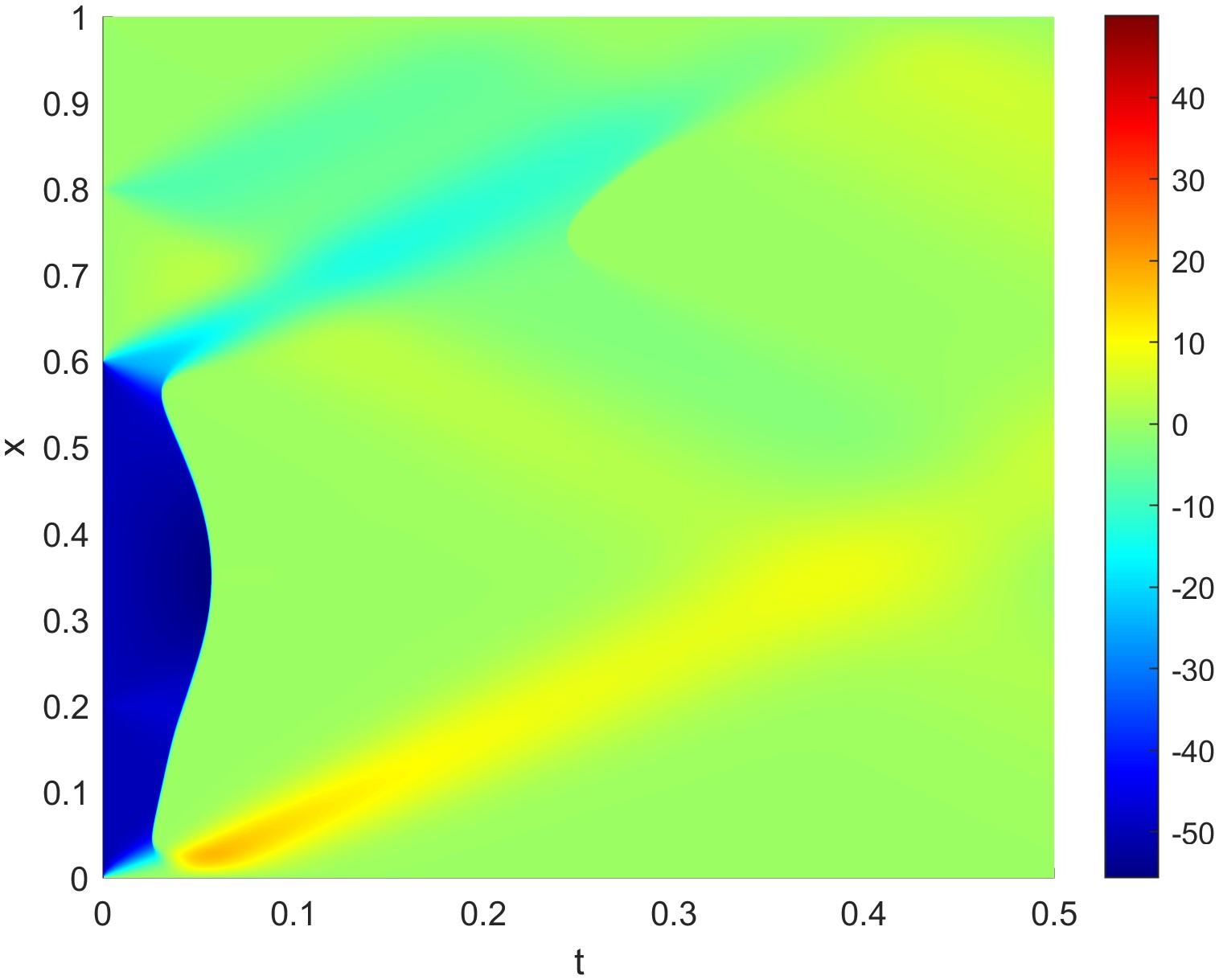}
    \end{subfigure}
    \caption{ Contact set (left) and velocity field (right) for example 2.} 
    \label{ex2}
\end{figure}

\noindent
\textbf{Acknowledgment}.
BM was supported by the Croatian Science Foundation under project number IP-2022-10-2962. ST was supported by the Science Fund of the Republic of Serbia, GRANT No TF C1389-YF, Project
title - FluidVarVisc. The authors are very grateful to the anonymous referees, their remarks significantly helped to improve the quality and clarity of the manuscript.

\Addresses

\end{document}